\documentclass[11pt]{amsart}
\pdfoutput=1

\usepackage{pdfsync}

\newcommand{\Mpres}[2]{\operatorname{Mon}\bigl\langle #1\:|\:#2 \bigr\rangle}
\newcommand{\Mgen}[1]{\operatorname{Mon}\bigl\langle #1 \bigr\rangle}
\newcommand{\Sgen}[1]{\operatorname{Sgp}\bigl\langle #1 \bigr\rangle}
\newcommand{\Ggen}[1]{\bigl\langle #1 \bigr\rangle}

\newcommand{\Gpres}[2]{\bigl\langle #1\:|\:#2 \bigr\rangle}
\newcommand{\Ipres}[2]{\operatorname{Inv}\bigl\langle #1\:|\:#2 \bigr\rangle}

\newcommand{\floor}[1]{\operatorname{}\lfloor #1 \rfloor}

\DeclareMathOperator{\Aut}{Aut}

\usepackage[leqno]{amsmath}
\usepackage[utf8]{inputenc}
\usepackage[english]{babel}

\usepackage{amssymb}
\usepackage{amsthm}
\usepackage{faktor}
\usepackage{mathtools}
\usepackage{amsfonts}
\usepackage{enumerate}
\usepackage{enumitem}
\usepackage{multicol}
\usepackage{float}
\usepackage{mathrsfs}
\usepackage{footnote}
\usepackage{hyperref}
\usepackage{comment}

\usepackage{tikz}
\usepackage{tikz-cd}
\usetikzlibrary{arrows,automata}
\usetikzlibrary{decorations.markings}
\usetikzlibrary{decorations.markings,arrows}
\usetikzlibrary{arrows,quotes}
\usetikzlibrary{arrows.meta}
\usetikzlibrary{lindenmayersystems,arrows.meta}
\usetikzlibrary{calc}

\tikzset{%
  >={Latex[width=2mm,length=2mm]},
            base/.style = {rectangle, rounded corners, draw=black,
                           minimum width=2cm, minimum height=0.8cm,
                           text centered, font=\sffamily},
  open/.style = {base, fill=white!30},
       undec/.style = {base, fill=red!30},
    dec/.style = {base, fill=green!30},
         process/.style = {base, minimum width=2.5cm, fill=orange!15,
                           font=\ttfamily},
}

\allowdisplaybreaks

\newcommand{\N}{\mathbb{N}}
\newcommand{\Z}{\mathbb{Z}}

\usepackage[capitalise]{cleveref}

\theoremstyle{plain}
\newtheorem{theorem}{Theorem}[section]
\newtheorem{prop}[theorem]{Proposition}

\newtheorem{lemma}[theorem]{Lemma}
\newtheorem{cor}[theorem]{Corollary}

\theoremstyle{definition}
\newtheorem{mydef}[theorem]{Definition}
\newtheorem{example}[theorem]{Example}
\newtheorem{question}[theorem]{Question}
\newtheorem{remark}[theorem]{Remark}

\title[Magnus submonoids and membership problems]{Magnus submonoids and membership problems \\ in one-relator, surface and hyperbolic groups}

\thanks{The research of the authors was supported by the EPSRC Fellowship grant EP/V032003/1 ‘Algorithmic, topological and geometric aspects of infinite groups, monoids and inverse semigroups’. }
\keywords{Membership problem, surface group, one-relator group, hyperbolic group, Magnus submonoid, prefix membership, distortion, graded monoid}
\subjclass[2020]{20F05, 20F38, 20F65, 20F67, 20F10; 20M05, 20M18}

\makeatletter
\newcommand{\leqnomode}{\tagsleft@true\let\veqno\@@leqno}
\newcommand{\reqnomode}{\tagsleft@false\let\veqno\@@eqno}
\makeatother
\begin{document}

\vspace{-8mm}

\maketitle

\vspace{-4mm}

\begin{center}
ISLAM FONIQI
\footnote{
TU Berlin, Institut für Mathematik, Straße des 17. Juni 136, 10623 Berlin, Germany. 
Email: \texttt{foniqi@math.tu-berlin.de}.
}
and
ROBERT D. GRAY
\footnote{School of Engineering, Mathematics, and Physics, University of East Anglia, Norwich NR4 7TJ, England.
Email \texttt{Robert.D.Gray@uea.ac.uk}.
}

\end{center}

\begin{abstract} 
Motivated by its applications to the word problem for one-relator inverse monoids, 
via results of Ivanov, Margolis, and Meakin (2001), we prove several decidability and undecidability results about the 
 submonoid membership problem in one-relator, surface and hyperbolic groups.  
The membership problem in free-by-cyclic one-relator groups is undecidable in general. We prove a general result that gives sufficient conditions under which membership in submonoids of certain free-by-cyclic one-relator groups is decidable. 
We apply this result to   
show that membership in various families of submonoids of surface groups (both orientable and non-orientable) is decidable.  
We study graded submonoids of one-relator groups, proving results that give conditions under which such a submonoid  has decidable membership with linear distortion. 
These results are then applied to  improve on a result of Margolis--Meakin--\v{S}unik by showing the prefix monoid of a surface group has linear distortion.
Our results significantly extend previously known positive results proving decidability of membership in (non group) submonoids of surface groups.
We also apply our general results, and other techniques, to show that the Magnus submonoid membership problem is decidable in several families of one-relator groups.
Here a Magnus submonoid of a one-relator group $\Gpres{A}{r=1}$, where every letter of $A$ appears in $r$, is a submonoid generated by a strict subset of $A \cup A^{-1}$.
In particular we show the Magnus submonoid membership problem is decidable
in surface groups, Baumslag--Solitar groups, and in certain free-by-cyclic one-relator groups.    
We also prove results about the related positivity problem (also called the quasi-Magnus problem) which asks whether membership in the submonoid generated by $A$ is decidable. 
We resolve a problem posed by McCammond and Meakin in 2006 by showing that there is a hyperbolic group $G$ generated by a finite set $A$ such that it is undecidable whether an element can be represented by a positive word on the generators $A$. We show in addition this hyperbolic group can be chosen to be residually finite. We do this by 
giving a new general method for constructing finitely presented groups with undecidable positivity problem.  
\end{abstract}

\vspace{-4mm}

\section{Introduction}\label{sec: IntroductionOLD}

Fundamental work of Magnus in the 1930s revealed that a key step in understanding the structure and properties of a one-relator group is given by studying the subgroups that are generated by subsets of the generating set of the group. These subgroups are now referred to as \emph{Magnus subgroups} in the literature. More precisely, a Magnus subgroup of a one-relator group $\Gpres{A}{r=1}$ is a subgroup generated by a subset of $A$ that omits one of the letters appearing in the defining relator word $r$. Magnus proved that such subgroups are always free, and showed that membership in Magnus subgroups is decidable \cite{Magnus30, Magnus1932}. 
As the trivial subgroup is a Magnus subgroup, in particular Magnus's results imply that one-relator groups have decidable word problem. 
Since then numerous other important results about Magnus subgroups have been proved for one-relator groups (and more generally for one-relator products) and they continue to play a key role in the development of the theory of this class of groups (see e.g. \cite{linton2025theory} and the introduction to \cite{Howie2025}). While membership in Magnus subgroups of one-relator groups is decidable, in contrast it is a well-known open problem whether membership is decidable in arbitrary finitely generated subgroups; see e.g. \cite[Question 2.8.7]{linton2025theory}. The subgroup membership problem for one-relator groups, also called the generalised word problem, has received serious attention in the literature and has been resolved in certain cases; see e.g.  
\cite{hruska2001towers,
kapovich2005genericity, 
lauer2013cubulating, mccammond2005coherence}.  

More generally one can ask about deciding membership in  finitely generated submonoids of a one-relator group. It has only quite recently been discovered that there are one-relator groups with undecidable submonoid membership problem; the first example was given in \cite{gray2020undecidability}, and then further examples have been constructed in \cite{nyberg2022diophantine} and \cite{Foniqi_Gray_Nyberg-Brodda_2025}. The motivation for studying this question comes in part from connections between this problem and the word problem for one-relator monoids and one-relator inverse monoids. These connections were established in work of Ivanov, Margolis and Meakin \cite{ivanov2001one} and also of Guba \cite{guba1997relationship}. In more detail, it is a longstanding open problem whether the word problem is decidable for one-relation monoids. Important work of Ivanov, Margolis and Meakin \cite{ivanov2001one} gives an approach to the word problem for one-relation monoids that goes via a class of monoids that lie between arbitrary monoids and groups, called inverse monoids. Their results show that to solve the word problem for one-relation monoids it would suffice to show that all one-relator inverse monoids of the form $M = \Ipres{a,b}{uv^{-1}=1} $ with $u, v \in \{a,b\}^+$ and $uv^{-1}$ a reduced word have decidable word problem. The word problem for $M = \Ipres{a,b}{uv^{-1}=1} $ is even open in the case that $uv^{-1}$ is a cyclically reduced word, and resolving this case would be an important first step towards resolving the general problem. The word problem for this class of inverse monoids has received attention and been solved in several instances in the papers \cite{dolinka2021new,
inam2025word,
margolis2005distortion}. This problem is closely related to the study of submonoids of one-relator groups via the following result which is a straightforward corollary of 
\cite[Theorem~3.1]{ivanov2001one}. 
\begin{theorem}
\cite[Theorem~3.1]{ivanov2001one}\label{AdianPositiveSubmonoid}
Let $M = \Ipres{a,b}{uv^{-1}=1}$ with $u, v \in A^+$ and $uv^{-1}$ cyclically reduced. Let $P = \Mgen{a,b}$ be the  submonoid of the one-relator group $G=\Gpres{a,b}{uv^{-1}=1}$ generated by the set $\{a,b\}$. If membership in $P \leq G$ is decidable then $M$ has decidable word problem. \end{theorem}

We note that  the word problem for one-relator inverse monoids with relators of the form $uv^{-1}$, with $u$ and $v$ both positive and $uv^{-1}$ a cyclically reduced word, is open. However, in general there do exist one-relator inverse monoids with undecidable word problem; see \cite{gray2020undecidability}.  

Theorem~\ref{AdianPositiveSubmonoid} provides motivation for studying the following problem which is one of the topics of the present article. Given a one-relator group $G=\Gpres{A}{w=1}$, is there an algorithm that takes a word over $A \cup A^{-1}$ and decides whether it can be written as a word in $A^+$, that is, whether it can be written as a positive word over the generators. Equivalently this question asks whether membership is decidable in the submonoid $\Mgen{A} \leq G$ generated by $A$. We call this the \emph{positivity problem}. 
More generally given a finitely presented group $\Gpres{A}{R}$ we call the monoid $\Mgen{A}$ the \emph{positive submonoid} of $G$ (with respect to the generating set $A$) and the positivity problem asks whether membership in the submonoid $\Mgen{A}$ of the group  $\Gpres{A}{R}$ is decidable. 
For general finitely presented groups this problem has been studied in the literature where it is called the \emph{quasi-Magnus problem}. For example, this problem was studied by Bill Boone in his Ph.D. dissertation of 1952 where he gave the first example of a finitely presented group for which this problem is undecidable. These examples with undecidable positivity problem (also called quasi-Magnus problem) were in fact an important ingredient in the construction of the first finitely presented groups with undecidable word problem; see e.g. 
\cite{stillwell1982word}. 

In addition to the positivity problem, we shall in this paper consider the natural generalisation of this problem to the study of membership in Magnus submonoids of one-relator groups. 
In more detail, let $G = \Gpres{A}{r=1}$ be a one-relator group. Given a subset $X$ of $A \cup A^{-1}$ such that there is a letter $a$ appearing in $r$ such that at least one of $a$ or $a^{-1}$ does not belong to $X$, then we call $M = \Mgen{X} \leq G$ a \emph{Magnus submonoid} of $G$. 
In particular, if every generator from $A$ appears in the word $r$, then a Magnus submonoid is simply one that is generated by some strict subset of $A \cup A^{-1}$.
Of course every Magnus subgroup is a Magnus submonoid, and also the positive submonoid $\Mgen{A}$ is a Magnus submonoid. 
Note that not every Magnus submonoid is contained in a Magnus subgroup e.g. this is the case for the submonoid $\Mgen{A}$.  We are interested in the question of whether membership in Magnus submonoids of one-relator groups is decidable. Of course a positive answer to this question would in particular show that membership in the submonoid $\Mgen{A}$ is decidable and thus resolve the open question of decidability of the word problem for the class of one-relator inverse monoids in the statement of Theorem~\ref{AdianPositiveSubmonoid} above.  

In the setting of \cref{AdianPositiveSubmonoid} the positive submonoid in question is in fact a one-relation monoid $\Mpres{a,b}{u=v}$ which naturally embeds in $\Gpres{a,b}{u=v}$ by a result of Adian \cite{adian1960embeddability}. The discussion above explains our motivation for studying the problem of deciding membership in this submonoid of this one-relator group. This question is part of a broader study of the membership problem in the positive submonoids of finitely presented groups in general. Moving away from the class of one-relator groups the study of positive submonoids of finitely presented groups arises naturally in other contexts. As explained above Boone gave the first examples of finitely presented groups for which membership in the positive submonoid, what he called the quasi-Magnus problem, is undecidable.   
There are other situations in geometric group theory where the positive submonoid of a finitely presented group plays an important role. For instance this is the case in the theory of Artin groups. By a result of Paris \cite{paris2002artin} the positive submonoid of any Artin group naturally embeds into that group. 
For certain families of Artin groups 
algorithmic questions about the Artin group can be translated into corresponding questions about the Artin monoid. This can be useful since Artin monoids are better behaved algorithmically than Artin groups; e.g they have nice normal forms \cite{brieskorn1972artin, michel1999note}. 
The membership problem for Artin monoids in their Artin groups is interesting since it is straightforward to show that if it is decidable for an Artin group then it solves the word problem for that Artin group, and the word problem for Artin groups remains open in general. 
We refer the reader to the survey article \cite{mccammond2017mysterious} for more background on Artin groups. 
Returning to finitely presented groups in general, there are some contexts where membership in the positive submonoid will be decidable. For instance, if the group is virtually free this will be the case since by a result of Benois \cite{Benois} 
free groups have decidable rational subset membership problem from which it follows that virtually free groups also do; see \cite{lohrey2015rational}. 
In 2006 McCammond and Meakin posed the natural question of whether this might hold more generally for hyperbolic groups; 
specifically, there they ask the following question: 

\begin{quote} (McCammond and Meakin (2006) \cite[Problem~1]{McCammondMeakin})
Can one decide whether an element in a hyperbolic group $G$ generated by $S$ is represented by a positive word in the generators?
\end{quote}
In Section~\ref{sec:hyperbolic} below we shall resolve this problem by giving 
a new general method for constructing finitely presented groups with undecidable positivity problem, and then applying it to give examples of hyperbolic groups for which the problem is undecidable. In addition we show that the hyperbolic group can be chosen to also be residually finite 
 still with the positivity problem being undecidable.

The investigation of the membership problem in Magnus submonoids of one-relator groups in this paper leads us to prove several general results showing that we can decide membership in particular submonoids of certain free-by-cyclic one-relator groups, including surface groups. Our results increase the known examples of submonoids of surface groups in which membership is decidable. The general submonoid membership problem 
remains open for surface groups. 
The submonoid membership problem for surface groups is a natural and well-motivated problem due to its connection with other open problems. For example, the submonoid membership problem is open for the special linear group $\mathrm{SL}(3,\Z)$; 
see the introduction to the paper \cite{semukhin2019reachability} for more background on membership problems in matrix groups. 
The group $\mathrm{SL}(3,\Z)$ is known to embed the hyperbolic orientable surface group of genus two \cite{long2011zariski} and since decidability of the submonoid membership problem is inherited by finitely generated subgroups, a solution to the problem for surface groups would shed light on the corresponding problem for $\mathrm{SL}(3,\Z)$. 
If the submonoid membership problem were undecidable for hyperbolic surface groups then it would be especially interesting due to the many examples of groups that have hyperbolic surface subgroups e.g.  infinite fundamental groups of  closed, hyperbolic, irreducible 3-manifolds \cite{KahnMarkovic}.
See \cite[Section~3.1]{Foniqi_Gray_Nyberg-Brodda_2025} for a more detailed discussion of the submonoid membership problem in surface groups.   
The submonoid membership problem for surface groups has been considered in the literature e.g. in the work of Margolis--Meakin--\v{S}unik who prove in \cite[Proposition 2.10]{margolis2005distortion} that for orientable surface groups membership in the prefix monoid is decidable. They do this by proving there is a quadratic upper distortion function for the prefix monoid; 
see Section~\ref{sec: Distortions} below for definitions of these terms. 

The general question of whether surface groups have decidable submonoid membership problem remains open. In contrast, the subgroup membership problem is known to be decidable in surface groups; see \cite{scott1978subgroups}. 
This can be used to show that the general submonoid membership problem for surface groups reduces to that for monoids generated by subsets which generate a finite index subgroup. Indeed, any subgroup of a surface group is either of finite index, or is free of infinite index. Hence if $X$ is a finite subset of a surface group $G$ and if the subgroup $H$ of $G$ generated by $X$ has infinite index in $G$, then the submonoid $M$ of $G$ generated by $X$ is contained in the finite rank free group $H$. 
Thus membership in the submonoid $M \leq G$ is decidable since membership in the subgroup $H \leq G$ is decidable and, by Benois' theorem \cite{Benois}, membership in the submonoid $M \leq H$ is decidable. 
Furthermore it follows that, since finite index subgroups of surface groups are again surface groups, the submonoid membership problem in surface groups reduces to considering just submonoids generated by finite subsets which are group generating sets for the surface group.

In this paper we shall prove both decidability and undecidability results. The main results of this article are the following:  
\begin{itemize} 
\item
In Theorem~\ref{thm: exponent_sum_one_relator2:NewCorrected} we give a general result that gives sufficient conditions under which membership in submonoids of certain free-by-cyclic one-relator groups is decidable.
Note that the membership problem in free-by-cyclic one-relator groups is undecidable in general.
\item 
We prove several new results showing that membership in various families of submonoids of surface groups (both orientable and non-orientable) is decidable; see 
\cref{thm:PositiveSubonoidsOfSurfaceGroups},
\cref{thm:surface:magnus:sub},
\cref{prop:LowRank},
\cref{thm:S2matchedNew} and
\cref{cor:powers}. 
These results significantly extend previously known positive results proving decidability of membership in (non group) submonoids of surface groups.
\item
We prove that the Magnus submonoid membership problem is decidable in surface groups (\cref{thm:MagnusSubmonoidsSurfaceGroupsGeneral}), Baumslag--Solitar groups $BS(m,n)$ (\cref{thm: magnus in BS groups}), and in the non-subgroup separable one-relator group of Burns, Karrass and Solitar 
\cite{Burns1987} (which is known to have undecidable submonoid membership problem by \cite{gray2020undecidability}); see \cref{prop:BurnsGroup}.
\item  
We resolve the problem posed by McCammond and Meakin in 2006 \cite[Problem~1]{McCammondMeakin}, discussed above, by showing that there is a hyperbolic group $G$ generated by $S$ such that it is undecidable whether an element can be represented by a positive word on the generators $S$; see \cref{thm:posUndecHyp}. We also show that the hyperbolic group can be chosen to have the additional property of being residually finite; see \cref{thm:ResFin}.
We do this by giving a general 
method in \cref{thm:positivity:fails:in:general} 
to constructing finitely presented groups with undecidable positivity problem. 
\item We improve on the abovementioned result of Margolis--Meakin--\v{S}unik for the prefix submonoids of surface groups, by showing that any surface group, both orientable and non-orientable ones, has decidable prefix membership problem, with a \emph{linear} distortion function; see \cref{thm:PMPDistortionSurface}.
We also prove linear distortion results for submonoids of other related classes of one-relator groups; see e.g. \cref{thm:S2matchedAdapted}. 
\item In \cref{thm:free:by:cyclic} we prove that the positivity problem is decidable for $2$-generator free-by-cyclic one-relator groups that have precisely one generator with exponent sum zero.     
The Magnus submonoid membership problem for general free-by-cyclic one-relator groups remains open. We discuss this problem at the end of Section~\ref{sec: Magnus submonoid membership problem in one-relator groups} and we prove a related undecidability result by showing in Theorem~\ref{thm:AutFreeGroupMonGen} that there is a finite subset $X$ of the free group $F_2$ of rank two and a single automorphism $\theta$ of $F_2$ such that membership in the submonoid of $F_2$ generated by the orbit of $X$ under $\theta$ is undecidable. 
\end{itemize}
We end the paper in Section~\ref{sec:HighDim} with some comments on the higher dimensional question of the submonoid membership problem in $3$-manifold groups. In particular we highlight the fact that there are many $3$-manifold groups for which this problem is undecidable, while in contrast the subgroup membership problem is known to be decidable for all $3$-manifold groups \cite{friedl2016membership}.

\section{Preliminaries}\label{Sec2-Preliminaries}
In this section we shall give some definitions, notation and results that will be needed in the rest of the article.    
For more details we refer the reader to 
\cite{magnus1966karrass, lyndon1977combinatorial} for combinatorial and geometric group theory and \cite{howie1995fundamentals} for semigroup and monoid theory.

\subsection{Group presentations}
We use $A^*$ to denote the \emph{free monoid} of all words over an alphabet $A$ including the empty word denoted by $\varepsilon$. By the \emph{length} $l(w)$ of a word $w \in A^*$ we mean the number of letters in that word. We denote by $F_A$ the \emph{free group} on the alphabet $A$ and use $\Gpres{A}{R}$ to denote the group defined by the presentation with generators $A$ and defining relators $R \subseteq F_A$. 
We typically write defining relations for a group presentation in the form $w=1$ or in the form $u=v$. 
When working with a presentation $\langle A \mid R \rangle$ of a group $G$, given two words $u$ and $v$ from $(A \cup A^{-1})^*$ we write $u \equiv v$ to mean that $u$ and $v$ are equal as words in the free monoid $(A \cup A^{-1})^*$, and write $u =_G v$, or often just $u=v$, to mean that $u$ and $v$ represent the same element of the group $G$. Given a subset~$X$ of a monoid we use~$\Mgen{X}$ to denote the submonoid generated by the set $X$ and given a subset $Y$ of a group we use $\Ggen{Y}$ to denote the subgroup generated by~$Y$.  

Given a monoid $M$ generated by a finite set $X$ for each element $m \in M$ we define
\[
|m|_X = \min\{ k \mid m = x_1 \ldots x_k, \; \mbox{for some} \; x_i \in X, \; i=1, \ldots k \}.
\]
Given a word $r \in (A \cup A^{-1})^*$ and a letter $a \in A$ we say that \emph{$a$ appears in $r$} if either $a$ or $a^{-1}$ is a letter in the word $r$. So for example the letters $a$ and $b$ both appear in the word $ab^{-1}$. As usual, we use $[a,b]$ to denote the commutator $aba^{-1}b^{-1}$.                

Given a group $G$ generated by a finite set $A$ we use $Cay_A(G)$ to denote the \emph{Cayley graph} of $G$ with respect to $A$, which has vertex set $G$ and edges between pairs $\{ g, ga \}$ where $g \in G$ and $a \in A$. We view $Cay_A(G)$ as a metric space with the usual metric on graphs where the distance between two vertices is the length of a shortest path between them.

\subsection{Submonoid membership problem}
Let~$M$ be a monoid finitely generated by a set~$A$ and let~$\phi:A^* \rightarrow M$ be the canonical homomorphism. 
Let $W$ be a finite subset of $A^*$ and let $T = \Mgen{W}$ be the submonoid of $M$ generated by the subset $\phi(W)$ of $M$.   
We say that \emph{membership in the submonoid $T \leq M$ is decidable} if there is an algorithm that solves the following decision problem:    
\begin{itemize} 
\item{\textsc{Input}:} A word $u \in A^*$   
\item{\textsc{Question}:} $\phi(u) \in \Mgen{W}$?   
  \end{itemize}
Also, we say that the monoid $M$ has \emph{decidable (uniform) submonoid membership problem} if   
there is an algorithm that solves the following decision problem: 
\begin{itemize} 
\item{\textsc{Input}:} A finite set of words~$W \subseteq A^*$ and a word~$w \in A^*$   
\item{\textsc{Question}:}~$\phi(w) \in \Mgen{W}$?   
  \end{itemize}
Whether or not a finitely generated monoid $M$ has decidable submonoid membership problem 
does not depend on the choice of finite generating set for the monoid $M$.    
When considering the submonoid membership problem in a group $G$ generated by a finite set $A$ the definitions are as above working with the finite monoid generating set $A \cup A^{-1}$ for the group $G$.   
  
The submonoid membership is well behaved when passing to finitely generated submonoids, in the following sense. If $M$ is a finitely generated monoid with decidable submonoid membership problem and if $T$ is a finitely generated submonoid of $M$ then $T$ also has decidable submonoid membership problem. Furthermore, for any submonoid~$S \leq T$ of~$T \leq M$, if membership in the submonoid $S \leq M$ is decidable then membership in the submonoid $S \leq T$ is decidable; see \cite[\S5]{lohrey2015rational}. 

The \emph{prefix monoid} of a one-relator group $\Gpres{A}{w=1}$ is the submonoid $P$ generated by the set of all prefixes of the word $w$. We say $\Gpres{A}{w=1}$ has decidable \emph{prefix membership problem} if membership in the prefix monoid is decidable. Note that the prefix monoid of a one-relator group depends on the choice of presentation for the group.        

We use $A(P_4)$ to denote the group defined by the presentation
\[
\langle a,b,c,d \mid ab=ba, bc=cb, cd=dc \rangle 
\]
which is the right-angled Artin group with underlying graph the path $P_4$ with four vertices $a,b,c,d$. It was proved in \cite[Theorem~7]{lohrey2008submonoid} that this group $A(P_4)$ contains a fixed finitely generated submonoid in which membership is undecidable.

\subsection{One-relator groups, the Magnus hierarchy and Magnus subgroups and submonoids}

Let $X$ be a finite set.
For a word $w \in (X \cup X^{-1})^*$ and an element $y \in X \cup X^{-1}$ we denote by $|w|_{y}$ the number of occurrences of $y$ in $w$.
For $t \in X$, the \emph{exponent sum} of $t$ in $w$ is the number $|w|_t - |w|_{t^{-1}}$.
Let $w\in (X \cup X^{-1})^*$ be a word 
such that $w$ contains a letter 
$t \in X$ with exponent sum zero. For $w$, we define a word $\rho_t(w)$ over the infinite alphabet
$$\Omega = \{x_l:\ x\in X\setminus\{t\}, l\in\mathbb{Z}\}$$
obtained from $w$ in the following way. 
For every letter $x^{\epsilon}$ of the word $w$ with $x \neq t$ upon writing  
$w \equiv w_1 x^{\epsilon} w_2$ rewrite letter $x^{\epsilon}$ to $x_{i}^{\epsilon}$, where $i$ is the exponent sum of $t$ in $w_1$. Then at the end delete every occurrence of $t$ from the resulting word. 
For example if $X = \{a,b,t\}$ then  
\[
\rho_t(b^{-1}t^{-1} a^{-2} t^2 b t^{-1} a) = b_0^{-1} a_{-1}^{-2} b_{1} a_0. 
\]
Then for each  $x\in X\setminus\{t\}$ let $\mu_x$ and $m_x$ be respectively the smallest and the greatest value of $j$ such that  $x_j$ actually appears in $\rho_t(w)$.
For instance, in the example in the displayed equation above we have $\mu_a = -1$ and $m_a = 0$.

The following result gives the well-known method for expressing certain one-relator groups as HNN extensions of one-relator groups with a shorter defining relator;  
see~\cite[Page 198]{lyndon1977combinatorial}. The formulation of the statement of the result here follows \cite[Proposition~7.1]{dolinka2021new}. 

\begin{prop}\label{pro:mol}
Let $w\in (X \cup X^{-1})^*$ be a word in which $t\in X$ has exponent sum zero such that $\rho_t(w)$ is cyclically reduced. Then 
the group $G=\Gpres{X}{w=1}$ is an HNN extension of the group 
$$
H = \Gpres{\Omega_w}{\rho_t(w)=1}
$$
where $\Omega_w=\{x_l:\ x\in X\setminus\{t\}, \mu_x\leq l\leq m_x\}$. The associated subgroups $A$ and $B$ in this extension
are free groups freely generated by $\Omega_w\setminus\{x_{m_x}:\ x\in X\setminus\{t\}\}$ and $\Omega_w\setminus
\{x_{\mu_x}:\ x\in X\setminus\{t\}\}$, respectively, with the isomorphism $\phi:x_i\mapsto x_{i+1}$ for all $x\in X\setminus\{t\}$
and $\mu_x\leq i<m_x$.
\end{prop}
We end this section by recalling the definition of Magnus submonoid from the introduction. 
   
\begin{mydef}
Let $G = \Gpres{A}{r=1}$ be a one-relator group. A \emph{Magnus submonoid} of $G$ is a submonoid generated by a subset $X \subseteq A \cup A^{-1}$ with the property that there is a letter $a$ appearing in $r$ such that $\{a,a^{-1}\}$ is not a subset of $X$. In particular, if every generator from $A$ appears in the word $r$, then a Magnus submonoid is simply one that is generated by some strict subset of $A \cup A^{-1}$. 
\end{mydef}

\section{Membership in submonoids of free-by-cyclic one-relator groups}\label{sec: Main Theorem}

There are free-by-cyclic one-relator groups with undecidable submonoid membership problem. In fact, the group $\Gpres{a,t}{a(tat^{-1}) = (tat^{-1})a}$ is a one-relator group, it is free-by-cyclic of the form $F_2 \rtimes_\phi \mathbb{Z}$ since it is isomorphic to the group $\langle x, y, t \mid txt^{-1} = xy, tyt^{-1} = y \rangle$, yet it has undecidable submonoid membership problem \cite{Burns1987, gray2020undecidability}. We note also that this group is a 3-manifold group; see \cite{niblo2001}. As discussed in the introduction, surface groups, which are also examples of free-by-cyclic one-relator groups, form an important class of one-relator groups for which the submonoid membership problem remains open.

In this section, we present some general results that identify sufficient conditions under which membership in certain submonoids of some free-by-cyclic one-relator groups is decidable. These general results will subsequently be applied to examples, including the membership problem for certain submonoids of surface groups, later in the article.

We shall now state and prove a general result that provides a sufficient condition for membership in a submonoid of certain one-relator groups to be decidable. The key condition employed in the proof is a max/min-style criterion, reminiscent of Brown's condition~\cite{Brown1987}. The max/min condition in the statement of Theorem~\ref{thm: exponent_sum_one_relator2:NewCorrected} will imply that the one-relator groups in the statement of that theorem are all free-by-cyclic. In Section~\ref{sec: Magnus submonoid membership problem in one-relator groups} below we shall apply Theorem~\ref{thm: exponent_sum_one_relator2:NewCorrected} to give a large family of submonoids of surface groups in which we can prove membership is decidable; see Theorem~\ref{thm:PositiveSubonoidsOfSurfaceGroups}. In particular that result shows that Theorem~\ref{thm: exponent_sum_one_relator2:NewCorrected} can be applied to finitely generated free-by-cyclic one-relator groups of the form $F \rtimes_\phi  \Z $ where $F$ is an infinite rank free group, as is the case for surface groups.

\begin{theorem}\label{thm: exponent_sum_one_relator2:NewCorrected}
Let $G = \Gpres{A, t}{w = 1}$ be a one-relator group with $\sigma_t(w) = 0$, where $A = \{a, b, c, \ldots\}$ is a finite set. Applying the Magnus rewriting process to $G$, we obtain  
\[
G = \Gpres{a_i, b_i, c_i, \ldots, t}{
\overline{w} = 1, \quad
t a_i t^{-1} = a_{i+1}, \quad
t b_i t^{-1} = b_{i+1}, \quad
t c_i t^{-1} = c_{i+1}, \ldots
}  
\]
where $i \in \mathbb{Z}$, and $\overline{w}$ is the word over $\{a_i, b_i, c_i, \ldots \mid i \in \mathbb{Z}\}^{\pm 1}$ obtained from $w$ by applying the Magnus procedure with respect to the letter $t$. 

Assume that for some letter $x \in A$, the element $x_{\mathrm{max}(x)}$ appears exactly once in the word $\overline{w}$, and $x_{\mathrm{min}(x)}$ also appears exactly once in $\overline{w}$, where 
\[
\mathrm{max}(x) = \max\{ i \mid x_i \text{ appears in } \overline{w} \}, 
\quad \text{and} \quad  
\mathrm{min}(x) = \min\{ i \mid x_i \text{ appears in } \overline{w} \}. 
\]
Let $w_1, \ldots, w_k \in FG(A, t)$ be elements each with $t$-exponent sum $\geq 0$. Then membership in the submonoid $M = \Mgen{w_1, \ldots, w_k} \leqslant G$ is decidable.
\end{theorem}

\begin{remark}
The condition in the statement of Theorem~\ref{thm: exponent_sum_one_relator2:NewCorrected}, that the words $w_1, \ldots, w_k \in FG(A, t)$ each have $t$-exponent sum $\geq 0$ is essential and cannot be omitted. Indeed, consider the group
\[
G = \Gpres{a,t}{tata^{-1}t^{-1}at^{-1}a^{-1} = 1}.
\]
Rewriting the relator with respect to $t$ yields the word $\overline{w} = a_1a_2^{-1}a_1a_0^{-1}$, which does satisfy the max/min condition stated in the theorem. However, it is known (see \cite{gray2020undecidability}) that this group contains a finitely generated submonoid for which the membership problem is undecidable.
\end{remark}

The rest of this section will be devoted to the proof of 
Theorem~\ref{thm: exponent_sum_one_relator2:NewCorrected}.
A key lemma we shall need to prove the theorem is the following. 
 
\begin{lemma}[{\cite[Corollary 6.10]{dolinka2021new}}]\label{lem:DG}
Let $X$ be a finite alphabet, and let
\[
G = F_X \; \ast_{t, \phi: P \rightarrow Q}
\]
be an HNN extension, where $P$ and $Q$ are both finitely generated. Then, for any finite subsets $W_0, W_1, \ldots, W_d$ and $W_1', \ldots, W_d'$ of $F_X$, the membership problem for the submonoid
\[
M = \Mgen{W_0 \cup W_1t \cup \ldots \cup W_dt^d \cup tW_1' \cup t^2W_2' \cup \ldots \cup t^d W_d'}
\]
is decidable.
\end{lemma}

A key observation is that, in the lemma, the subgroups $P$ and $Q$ are arbitrary subgroups of the free group $F_X$, and $\phi$ is an isomorphism between them. In particular, there is no assumption that $P$ is generated by a subset of $X$.

We now fix some notation that will remain in place throughout the rest of this section. As in the statement of Theorem~\ref{thm: exponent_sum_one_relator2:NewCorrected} we let $G = \Gpres{A, t}{w = 1}$ be a one-relator group with $\sigma_t(w) = 0$. Let $A = \{a, b, c, \ldots\}$. Applying Magnus rewriting to $G$, we obtain
\[
G = \Gpres{a_i, b_i, c_i, \ldots, t}{\overline{w} = 1, \quad
t a_i t^{-1} = a_{i+1}, \quad
t b_i t^{-1} = b_{i+1}, \quad
t c_i t^{-1} = c_{i+1}, \ldots}
\]
where $i \in \mathbb{Z}$, and $\overline{w}$ is the word over $\{a_i, b_i, c_i, \ldots \mid i \in \mathbb{Z} \}^{\pm 1}$ obtained from $w$ by applying the Magnus procedure with respect to the letter $t$. We assume that, for some letter $x \in A$, the letter $x_{\mathrm{max}(x)}$ appears exactly once in the word $\overline{w}$, and $x_{\mathrm{min}(x)}$ also appears exactly once in $\overline{w}$, where
\[
\mathrm{max}(x) = \max\{ i \mid x_i \text{ appears in } \overline{w} \}, 
\quad \text{and} \quad 
\mathrm{min}(x) = \min\{ i \mid x_i \text{ appears in } \overline{w} \}.
\]
Let $w_1, \ldots, w_k \in F_{A \cup \{ t \}}$ be words all having $t$-exponent sum $\geq 0$. The rest of this section will be devoted to the proof that membership in the submonoid $M = \Mgen{w_1, \ldots, w_k} \leqslant G$ is decidable. We shall divide the proof into a sequence of lemmas.

In the following statement, we use $\sigma_i(\overline{w})$ to denote the word obtained by adding $i$ to the subscript of every letter appearing in $\overline{w}$. For example, if $\overline{w} \equiv [x_0, z_0]y_0 y_1^{-1}$, then $\sigma_2(\overline{w}) \equiv [x_2, z_2]y_2y_3^{-1
 }$.

\begin{lemma}\label{lem:Gpres}
The group $G$ admits the infinite presentation
\[
G = \Gpres{a_i, b_i, c_i, \ldots, t}{
\sigma_i(\overline{w}) = 1 \; (i \in \mathbb{Z}), \quad
t a_i t^{-1} = a_{i+1}, \quad
t b_i t^{-1} = b_{i+1}, \quad
t c_i t^{-1} = c_{i+1}, \ldots}
\]
where $\sigma_i(\overline{w})$ is the word obtained by adding $i$ to the subscript of every letter appearing in $\overline{w}$.
\end{lemma}

\begin{proof}
All the relations $\sigma_i(\overline{w}) = 1$ hold, since they are obtained from the relation $\overline{w} = 1$ by conjugating with appropriate powers of $t$.
\end{proof}

A central idea in the proof of Theorem~\ref{thm: exponent_sum_one_relator2:NewCorrected} will be that, for every interval $[n, m]$ of integers with $n < m$, we can extract the corresponding subset of generators and relations from the infinite presentation above to obtain a finite presentation. This motivates the following definition.

\begin{mydef}
For any pair of integers $n < m$, we define
\[
L_{[n,m]} = \Gpres{A_{[n,m]}, t }{
\sigma_i(\overline{w}) = 1 \; (n \leq i \leq m), \quad C_{[n,m]} }
\]
where $A_{[n,m]}$ is the set of all letters that appear in the set of words $\{ \sigma_i(\overline{w}) = 1 \mid n \leq i \leq m \}$, and $C_{[n,m]}$ is the set of all relations of the form $t x_i t^{-1} = x_{i+1}$ such that both $x_i \in A_{[n,m]}$ and $x_{i+1} \in A_{[n,m]}$, where $x \in \{a, b, c, \ldots\}$.
\end{mydef}

\begin{example}
In the case of the surface group
\[
  S_2 =
  \langle a,b,c,d \mid aba^{-1}b^{-1}cdc^{-1}d^{-1}=1 \rangle =
  \langle a,b,c,t \mid aba^{-1}b^{-1}ctc^{-1}t^{-1}=1 \rangle,
\]
Magnus rewriting with respect to $t$ gives the word $[a_0,b_0]c_0c_1^{-1}$ and then an example of such an interval presentation is the following:
\begin{align*}
L_{[0,2]} = \bigl\langle
& a_0, a_1, a_2,\ 
b_0, b_1, b_2,\ 
c_0, c_1, c_2, c_3, \ t \ \mid \ \\
& [a_0, b_0]c_0 = c_1,\quad 
[a_1, b_1]c_1 = c_2,\quad 
[a_2, b_2]c_2 = c_3,\quad \\
& ta_0t^{-1} = a_1,\quad 
ta_1t^{-1} = a_2,\quad 
tb_0t^{-1} = b_1,\quad 
tb_1t^{-1} = b_2,\quad \\ 
& tc_0t^{-1} = c_1,\quad 
tc_1t^{-1} = c_2, \quad  
tc_2t^{-1} = c_3 \quad  
\bigr\rangle
\end{align*}
This presentation corresponds to the interval $[0,2]$, where we use the index $i$ to label the relation $[a_i, b_i]c_i = c_{i+1}$.  
\end{example}

\begin{lemma}\label{lem:isomL} 
The group $L_{[n,m]}$ is isomorphic to $G$ for all $n < m$. 
In fact, the natural injective map from the generators of $L_{[n,m]}$ to the generators of $G$ extends to an isomorphism.
\end{lemma}

\begin{proof}
We begin with the presentation 
\[\Gpres{A_{[n,m]}, t }{ \sigma_i(\overline{w}) = 1 \; (n \leq i \leq m), \quad C_{[n,m]} } \]
for $L_{[n,m]}$, where $n < m$. Then for each letter $x_i$ in $A_{[n,m]}$ we add new generators $x_i$ for all $i \in \mathbb{Z}$ such that $x_i \not\in A_{[n,m]}$ and at the same time we add in all the relations $tx_i t^{-1} = x_{i+1}$ for $i \in \mathbb{Z}$ that do not already belong to the set $C_{[n,m]}$. These relations allow us to express each of the new generators $x_i$ we have added in terms of the original generators of $L_{[n,m]}$, thus this is a Tietze transformation that does not change the group defined by the presentation.   
Next, we add the infinite family of relations $\sigma_i(\overline{w}) = 1$ for all $i \in \mathbb{Z}$, each of which is now a consequence of the relation $\overline{w} =1$ under conjugation by powers of $t$. After performing these Tietze transformations we obtain the presentation for $G$ given in the statement of Lemma~\ref{lem:Gpres}. 
The fact that the natural injective map from the generators of $L_{[n,m]}$ to the generators of $G$ extends to an isomorphism follows from the fact that the only Tietze transformations carried out have been to add redundant generators and redundant relations.  
\end{proof}

In the next two results we invoke the max/min conditions from the statement of Theorem~\ref{thm: exponent_sum_one_relator2:NewCorrected}.

\begin{lemma}\label{lem:free:group}
For all $n < m$, the group
\[
H_{[n,m]} = \Gpres{A_{[n,m]}}{\sigma_i(\overline{w}) = 1 \; (n \leq i \leq m)}
\]
is a free group.
\end{lemma}

\begin{proof}
Up to relabelling, we may assume that the letter $a \in A$ satisfies the following: $a_{\mathrm{max}(a)}$ appears exactly once in the word $\overline{w}$, and $a_{\mathrm{min}(a)}$ appears exactly once in $\overline{w}$, where
\[
\mathrm{max}(a) = \max\{ i \mid a_i \text{ appears in } \overline{w} \}, \quad
\mathrm{min}(a) = \min\{ i \mid a_i \text{ appears in } \overline{w} \}.
\]

Since $a_{\mathrm{max}(a)}$ appears exactly once in $\overline{w}$, it follows that for all $i \in \mathbb{Z}$, the letter $\sigma_i(a_{\mathrm{max}(a)}) = a_{\mathrm{max}(a) + i}$ appears exactly once in the word $\sigma_i(\overline{w})$, for all $n \leq i \leq m$.

We now perform a sequence of Tietze transformations to eliminate the generators
\[
\{ a_{\mathrm{max}(a) + n}, \ldots, a_{\mathrm{max}(a) + m} \},
\]
starting with $\sigma_m(\overline{w})$ to eliminate $a_{\mathrm{max}(a) + m}$ and remove the relator $\sigma_m(\overline{w})$. We then proceed inductively, working down one generator at a time, until we use $\sigma_n(\overline{w})$ to eliminate $a_{\mathrm{max}(a) + n}$ and remove the relator $\sigma_n(\overline{w})$. This process yields a presentation with the generating set
\[
A_{[n,m]} \setminus \{ a_{\mathrm{max}(a) + n}, \ldots, a_{\mathrm{max}(a) + m} \},
\]
and no remaining relators. Hence, $H_{[n,m]}$ is a free group on this set of generators.
\end{proof}
Note that the previous proof relied only on the max condition and did not make use of the min condition.

\begin{remark}
The max (or min) condition used in proof of Lemma~\ref{lem:free:group}  is necessary. Consider the example
\[
\Gpres{a,t}{a a t a t a t^{-2}=1},
\]
which, working with respect to the generator $t$, rewrites to
\[
\Gpres{a_0, a_1, a_2}{a_0^2 a_1 a_2^2 = 1},
\]
a presentation of a free group. However, when we consider a shift of this presentation, for example:
\[
\Gpres{a_0, a_1, a_2, a_3}{a_0^2 a_1 a_2^2 = 1,\ a_1^2 a_2 a_3^2 = 1},
\]
we find that eliminating the generator $a_2$ yields a group isomorphic to the amalgamated free product $ \langle a_0 \rangle \ast_{a_0^2 = a_3^2 a_1^2 a_3^2 a_1} \langle a_1, a_3 \rangle$, which is not a free group. 
To see that this is not a free group it suffices (see e.g. \cite[Theorem~1.1(a)]{HOUCINE}) to verify that $a_0^2$ is not primitive in the free group $F_{a_0}$, which is clear, and also that $a_3^2 a_1^2 a_3^2 a_1$ is not primitive in $F_{\{a_1, a_3\}}$. For the latter, consider the group $ \Gpres{a_1, a_3}{a_3^2 a_1^2 a_3^2 a_1=1} $ and make the substitution $a_1 = a_3^{-2} b$ to obtain the isomorphic group  
\[
\Gpres{b, a_3}{b a_3^{-2} b b =1} = 
\Gpres{b, a_3}{b^3 = a_3^2 },  
\]
which is a torus knot group that is not a free group. Hence $ \Gpres{a_1, a_3}{a_3^2 a_1^2 a_3^2 a_1=1} $ is not a free group and thus the word $a_3^2 a_1^2 a_3^2 a_1$ is not primitive, as required.  
This illustrates that the max (or min) condition is essential in Lemma~\ref{lem:free:group} to ensure that the groups $H_{[n,m]}$ are free.
\end{remark}

\begin{lemma}\label{lem:keyforHNN} 
Fix some $n < m$. Let
\[
H_{[n,m]} = \Gpres{A_{[n,m]}}{ \sigma_i(\overline{w}) = 1 \; (n \leq i \leq m)}.
\]
Define
\[
P = \Ggen{A_{[n,m-1]}} \leq H_{[n,m]},  
\quad \text{and} \quad
Q = \Ggen{A_{[n+1,m]}} \leq H_{[n,m]}.
\]
Then
\[
P \cong H_{[n,m-1]} 
\quad \text{and} \quad 
Q \cong H_{[n+1,m]},
\]
and the map $\phi: P \rightarrow Q$ defined by
\[
x_i \mapsto x_{i+1}, \quad \text{for all } x \in A_{[n,m-1]},
\]
is an isomorphism between $P$ and $Q$.
\end{lemma}

\begin{proof}
Up to relabelling, we may assume that the letter $a \in A$ satisfies that $a_{\mathrm{max}(a)}$ appears exactly once in the word $\overline{w}$, and $a_{\mathrm{min}(a)}$ appears exactly once in $\overline{w}$, where
\[
\mathrm{max}(a) = \max\{ i \mid a_i \text{ appears in } \overline{w} \}, 
\quad \text{and} \quad  
\mathrm{min}(a) = \min\{ i \mid a_i \text{ appears in } \overline{w} \}.
\]
We can use the relation $\sigma_m(\overline{w})$
to eliminate the generator $a_{\mathrm{max}(a) + m}$ and remove the relator $\sigma_m(\overline{w})$ from the presentation. This yields
\[
H_{[n,m]} = \Gpres{A_{[n,m]} \setminus \{ a_{\mathrm{max}(a) + m} \}}{\sigma_i(\overline{w}) = 1 \; (n \leq i \leq m-1)}.
\]
It follows that
\[
P = \Ggen{A_{[n,m-1]}} = 
\Gpres{A_{[n,m-1]}}{\sigma_i(\overline{w}) = 1 \; (n \leq i \leq m-1)} 
\cong H_{[n,m-1]}.
\]
Similarly, using the relation
$\sigma_n(\overline{w})$
we may eliminate the generator $a_{\mathrm{min}(a) + n}$ and remove the relator $\sigma_n(\overline{w})$, giving
\[
H_{[n,m]} = \Gpres{A_{[n,m]} \setminus \{ a_{\mathrm{min}(a) + n} \}}{\sigma_i(\overline{w}) = 1 \; (n+1 \leq i \leq m)}.
\]
It then follows that
\[
Q = \Ggen{A_{[n+1,m]}} = 
\Gpres{A_{[n+1,m]}}{\sigma_i(\overline{w}) = 1 \; (n+1 \leq i \leq m)} 
\cong H_{[n+1,m]}.
\]
It is clear that $\phi$ defines an isomorphism from \( P \) to \( Q \), since the respective presentations differ only by a uniform shift of indices in the generators and relators.
\end{proof}

Note that in the previous lemma, all of the groups $H_{[n,m]}$, $P$, and $Q$ are free groups. However, none of them is free on the set of generators given in the presentations that define them in the statement of the lemma (that is, there exist relations among the generators).
Combining the previous lemma with the earlier lemmas in this section we can now express the group $G$ as an HNN extension of the free group $H_{[n,m]}$ with respect to the isomorphism  $\phi: P = \Ggen{A_{[n,m-1]}} \rightarrow Q = \Ggen{A_{[n+1,m]}}$.

\begin{lemma}\label{lem:BigHNNExtension} 
Let $n, m \in \mathbb{Z}$ with $n < m$. Then
\[
G \cong L_{[n,m]}   
= \Gpres{A_{[n,m]}, t }{\sigma_i(\overline{w}) = 1 \; (n \leq i \leq m), \quad C_{[n,m]}}
\]
is isomorphic to the HNN extension of the free group $H_{[n,m]}$ with respect to the isomorphism
\[
\phi: P = \Ggen{A_{[n,m-1]}} \rightarrow Q = \Ggen{A_{[n+1,m]}}
\]
between finitely generated subgroups, where $\phi$ is defined by
\[
x_i \mapsto x_{i+1}, \quad \text{for all } x \in A_{[n,m-1]}.
\]
\end{lemma}

\begin{proof}
It follows from Lemma~\ref{lem:keyforHNN} that this is a well-defined HNN extension. The result then follows immediately from Lemma~\ref{lem:isomL} and the relevant definitions.
\end{proof}

The significance of the previous lemma is that it expresses $L_{[n,m]}$ as an HNN extension of a free group, where the associated map $\phi$ is an isomorphism between finitely generated subgroups of that free group. Consequently, Lemma~\ref{lem:DG} may be applied to $L_{[n,m]}$ using this specific decomposition of $L_{[n,m]}$ as an HNN extension. When we apply Lemma~\ref{lem:DG} to the group $L_{[n,m]}$, we obtain the following result:

\begin{lemma}\label{lem:TheGoodOne:General}
Let $n, m \in \mathbb{Z}$ with $n < m$. Let
\[
G \cong L_{[n,m]}   
= \Gpres{A_{[n,m]}, t }{\sigma_i(\overline{w}) = 1 \; (n \leq i \leq m), \quad C_{[n,m]}}.
\]
Let $W_0, W_1, \ldots, W_d$ and $W_1', \ldots, W_d'$ be finite 
subsets of $(A_{[n,m]}^{\pm 1})^\ast$.
Then the membership problem for the submonoid
\[
M = \Mgen{W_0 \cup W_1 t \cup \ldots \cup W_d t^d \cup t W_1' \cup t^2 W_2' \cup \ldots \cup t^d W_d'}
\]
is decidable.
\end{lemma}
\begin{proof}
It follows from the definitions together with Lemma~\ref{lem:free:group} that the subgroup of $L_{[n,m]}$ generated by $A_{[n,m]}$ is a free group (although it is not free on the generating set $A_{[n,m]}$ itself). By Lemma~\ref{lem:BigHNNExtension}, $L_{[n,m]}$ is an HNN extension of this free group, with respect to an isomorphism between finitely generated subgroups. Thus, the hypotheses of Lemma~\ref{lem:DG} are satisfied and the result now follows by applying that lemma. 
\end{proof}

We are now in a position to prove Theorem~\ref{thm: exponent_sum_one_relator2:NewCorrected}. The basic idea of the proof is that, given a finite generating set for a submonoid of $G$ satisfying the hypotheses of the theorem, we can choose an interval $[n,m]$ sufficiently large such that, when working in $L_{[n,m]}$, the generating set can be expressed in the form required by Lemma~\ref{lem:TheGoodOne:General}. The theorem then follows directly from an application of that lemma.

\begin{proof}[Proof of Theorem~\ref{thm: exponent_sum_one_relator2:NewCorrected}]
Let $w_1, \ldots, w_k \in F_{A \cup \{t \}}$ all have $t$-exponent sum $\geq 0$. Working with respect to the infinite presentation
\[
G = \Gpres{a_i, b_i, c_i, \ldots, t}{
\overline{w} = 1,\quad
t a_i t^{-1} = a_{i+1},\quad
t b_i t^{-1} = b_{i+1},\quad
t c_i t^{-1} = c_{i+1}, \ldots
},
\]
we may write each $w_l$ in the form $w_l =_G t^{j_l} u_{l}$, where $u_{l}$ is a word in $(\{ a_i, b_i, c_i, \ldots \}^{\pm 1})^*$. By the assumptions of the theorem, each $j_l$ is greater than or equal to $0$. Let $F$ denote the finite set of generators from $\{ a_i, b_i, c_i, \ldots \}$ that appear in the finite collection of words $\{ u_{l}: 1 \leq l \leq k \}$. Choose integers $n < m$ such that $F \subseteq A_{[n,m]}$, the generating set of the finite presentation of $L_{[n,m]}$. By \cref{lem:isomL}  it follows that the equality $w_l =_{L_{[n,m]}} t^{j_l} u_{l}$ holds in $L_{[n,m]}$ for all $l$.
Now, by Lemma~\ref{lem:TheGoodOne:General}, membership in the submonoid
\[
M = \Mgen{w_1, \ldots, w_k} = \Mgen{t^{j_1} u_{1}, \ldots, t^{j_k} u_{k}} \leqslant L_{[n,m]}
\]
is decidable. Hence by \cref{lem:isomL} it follows that membership in $M = \Mgen{w_1, \ldots, w_k} \leqslant G$ is decidable.
\end{proof}

\begin{remark} 
Note that all steps involved in the proof of \cref{thm: exponent_sum_one_relator2:NewCorrected} are effective. In particular, we can explicitly write down the finite presentation $L_{[n,m]}$ and compute the expressions of the form $t^j u$ required in the proof. So the proof of \cref{thm: exponent_sum_one_relator2:NewCorrected} actually shows that there is a uniform algorithm that takes any one-relator group  and any collection of words $w_1, \ldots, w_k$ satisfying the conditions in the statement of Theorem~\ref{thm: exponent_sum_one_relator2:NewCorrected}, and any word $\tau$ over $(A \cup \{t\})^{\pm 1}$, and decides whether or not $\tau$ represents an element in $\Mgen{w_1, \ldots, w_k}$. 
\end{remark}

It is natural to ask whether 
\cref{thm: exponent_sum_one_relator2:NewCorrected} can be generalised in the following way.

\begin{question}
  Let~$G = \Gpres{A, t}{w = 1}$ be a one-relator group with~$\sigma_t(w) = 0$. Suppose that applying the Magnus procedure with respect to~$t$ yields a free group. If~$w_1, \ldots, w_k \in F_{A \cup \{ t \}}$ all have $t$-exponent sum at least~$0$ (or dually, at most~$0$), is the submonoid membership problem decidable in $M = \Mgen{w_1, \ldots, w_k} \leqslant G$?    
\end{question}

\section{Membership in surface groups and Magnus submonoids of one-relator groups}\label{sec: Magnus submonoid membership problem in one-relator groups} 

In this section we first show how Theorem~\ref{thm: exponent_sum_one_relator2:NewCorrected} can be applied to show membership can be decided in a large family of finite generated submonoids of surface groups. Then we shall go on to consider the membership problem in Magnus submonoids of one-relator groups proving that it is decidable for various families.

\subsection{Surface Groups}

We can apply \cref{thm: exponent_sum_one_relator2:NewCorrected} to prove the following result concerning surface groups:

\begin{theorem}\label{thm:PositiveSubonoidsOfSurfaceGroups} 
Let $S_g$ be the genus-$g$ orientable surface group: 
\[
S_g = \Gpres{a_1, \ldots, a_g, b_1, \ldots, b_g}{[a_1,b_1]\cdots[a_g,b_g]= 1}.
\]
Let $t \in \{ a_1, \ldots, a_g, b_1, \ldots, b_g \}$ be any generator. If~$\,W$ is a finite set of words in the free group on 
$\{ a_1,\ldots,a_g,b_1,\ldots,b_g \}$ 
such that each word in $W$ has $t$-exponent sum at least $0$, then membership in the submonoid $\Mgen{W} \leq S_g$ is decidable.
\end{theorem}

\begin{proof} 
Since permuting the subscripts cyclically gives an isomorphic group, and interchanging $a_i$ and $b_i$ for all $i$ gives an isomorphic group, we may assume without loss of generality that $t = a_1$. Rewrite the presentation as  
\[
S_g = \Gpres{t, a_2, \ldots, a_g, b_1, \ldots, b_g}{[t,b_1][a_2,b_2]\cdots[a_g,b_g] = 1}.
\]
The defining relator may be rewritten as 
\[
tb_1t^{-1} = [b_g,a_g]\cdots[b_2,a_2] b_1.
\]
Applying Magnus rewriting will rewrite the defining relation as 
\[
b_{1,1} = [b_{g,0},a_{g,0}]\cdots[b_{2,0},a_{2,0}] b_{1,0},
\]
where we use a pair of subscripts rather than a double subscript for clarity. This yields the following infinite presentation for the group $S_g$ :
\begin{align*}
\operatorname{Gp}\bigl\langle t,\ a_{2,i}, \ldots, a_{g,i},\ b_{1,i}, \ldots, b_{g,i}\ (i \in \mathbb{Z})\ \mid \ 
& b_{1,1}^{-1} [b_{g,0},a_{g,0}]\cdots[b_{2,0},a_{2,0}] b_{1,0} = 1, \\
& t a_{m,n} t^{-1} = a_{m,n+1}, \\
& t b_{p,n} t^{-1} = b_{p,n+1} \quad (2 \leq m \leq g,\ 1 \leq p \leq g,\ n \in \mathbb{Z}) 
\bigr\rangle
\end{align*}
The word $b_{1,1}^{-1} [b_{g,0},a_{g,0}]\cdots[b_{2,0},a_{2,0}] b_{1,0}$ satisfies the max–min condition of \cref{thm: exponent_sum_one_relator2:NewCorrected} (with $b_{1,1}$ and $b_{1,0}$ appearing exactly once), so we can apply that theorem to conclude the result.
\end{proof}

\begin{remark}
The previous theorem generalizes the result proved in~\cite{margolis2005distortion} that the surface group $S_g$ has a decidable prefix membership problem—since every prefix of the defining relator clearly has $a_1$-exponent sum at least~$0$. In fact, for any cyclic conjugate of the defining relator, \cref{thm:PositiveSubonoidsOfSurfaceGroups} shows that the prefix membership problem remains decidable, thereby also recovering the result~\cite[Theorem 5.10, Example 5.11]{dolinka2021new}.
\end{remark}

As a corollary, we also obtain the following:

\begin{cor}
Let 
\[
S_g = \Gpres{a_1, \ldots, a_g, b_1, \ldots, b_g}{[a_1,b_1] \cdots [a_g,b_g] = 1}.
\]
Then for any finite set of elements $x_1, \ldots, x_k \in S_g$, there exists a choice of signs $\epsilon_i \in \{-1,  +1\}$ such that membership in the submonoid $\Mgen{x_1^{\epsilon_1}, \ldots, x_k^{\epsilon_k}}$ is decidable.
\end{cor}

This corollary has a connection to a notion arising in the theory of left-orderable groups. As explained, for example, in \cite[Theorem~1.48]{clay2016ordered}, a group $G$ is left-orderable if and only if for every finite subset $\{x_1, \ldots, x_n\} \subseteq G \setminus \{1\}$, there exist signs $\epsilon_i = \pm 1$ such that the identity element does not lie in the subsemigroup $\Sgen{x_1^{\epsilon_1}, \ldots, x_n^{\epsilon_n}}$.

We now turn our attention to Magnus submonoids of one-relator groups. First, we apply the results above to show that this problem is decidable for surface groups, and then later in this section, we shall extend our analysis to certain other families of one-relator groups, including the Baumslag--Solitar groups and some $2$-generator free-by-cyclic one-relator groups.

As already discussed in the introduction, the general question of deciding membership in Magnus submonoids of one-relator groups, however, remains open -- and is likely to be difficult as it would imply solutions to other challenging open problems about the word problem for one-relator inverse monoids; see Theorem~\ref{AdianPositiveSubmonoid} and the discussion in the introduction above.

\begin{theorem}\label{thm:MagnusSubmonoidsSurfaceGroupsGeneral}
The membership problem in Magnus submonoids of orientable surface groups  
\[
S_g = \Gpres{a_1, \ldots, a_g, b_1, \ldots, b_g}{[a_1,b_1]\ldots[a_g,b_g]= 1}
\]
is decidable. Furthermore, membership in Magnus submonoids of the non-orientable surface groups 
\[
\mathcal{N}_g = \Gpres{a_1, \ldots, a_g}{a_1^2 \ldots a_g^2= 1}
\]
is also decidable.
\end{theorem}

\begin{proof}
Let $M$ be a Magnus submonoid of $S_g$. Then there is some generator $x$ such that either $x$ or $x^{-1}$ does not belong to the generating set. Hence, the generating set for $M$ satisfies that every generator has $x$-exponent sum either $\geq 0$ or $\leq 0$. The result then follows by an application of \cref{thm:PositiveSubonoidsOfSurfaceGroups} or its dual.

Now let $M$ be a Magnus submonoid of $\mathcal{N}_g$. If $g=2$, then $\mathcal{N}_g$ is the fundamental group of the Klein bottle, which is virtually abelian and thus has decidable 
submonoid membership problem (see e.g. \cite{lohrey2015rational}); 
in particular, membership in Magnus submonoids is decidable. Hence, assume $g \geq 3$.

Let $X$ be the subset of $\{ a_1, a_1^{-1}, \ldots, a_g, a_g^{-1} \}$ generating $M$. If all positive letters $a_i$ are included in $X$, then $M = \mathcal{N}_g$ since every inverse can be expressed positively. Dually, if all negative letters $a_i^{-1}$ are included in $X$, then by the inverse of the defining relator we can generate all positive letters, so again $M = \mathcal{N}_g$. 

Therefore, there must exist distinct indices $i, j$ such that $a_i \in X$ but $a_i^{-1} \notin X$, while $a_j^{-1} \in X$ but $a_j \notin X$. By cyclically permuting the generators (an isomorphism), we may assume $a_1 \in X$ but $a_1^{-1} \notin X$, and $a_j^{-1} \in X$ but $a_j \notin X$ for some $j \neq 1$. Suppose first that $j \neq g$, and perform the linear substitution $a_j = x a_1^{-1}$. Consider the isomorphic one-relator group
\[
\Gpres{
a_1, \ldots, a_{j-1}, x, a_{j+1}, \ldots, a_g
}{
a_1^2 a_2^2 \ldots a_{j-1}^2 (x a_1^{-1})^2 a_{j+1}^2 \ldots a_g^2 = 1
}.
\]
Since $j \neq g$, this relator is cyclically reduced, and has $a_1$-exponent sum zero. Applying Magnus rewriting with respect to $a_1$ yields
\[
a_{2,2}^2 \ldots a_{j-1,2}^2 x_2 x_1 a_{j+1,0}^2 \ldots a_{g,0}^2,
\]
which satisfies the max–min property from the statement of \cref{thm: exponent_sum_one_relator2:NewCorrected} with respect to the symbol $x$.

In this new presentation, the generators in $X$ include $a_j^{-1} = a_1 x^{-1}$, $a_1$, and some subset of $\{ a_2, a_2^{-1}, \ldots, a_g, a_g^{-1} \}$. Observe that all these generators have $a_1$-exponent sum at least zero. Hence, since the max–min condition
from the statement of \cref{thm: exponent_sum_one_relator2:NewCorrected} holds, we can apply that result to complete the proof.

If $j = g$, then we argue similarly using the substitution $a_j = a_1^{-1} x$. The defining relator remains cyclically reduced for $g \geq 3$. This completes the proof.
\end{proof}

It is worth noting that not every Magnus submonoid is contained in a Magnus subgroup, and Magnus submonoids are typically not free monoids. The exact structure of Magnus submonoids is interesting and worthy of future study.

We can generalise Theorem \ref{thm:MagnusSubmonoidsSurfaceGroupsGeneral} slightly to obtain the following. 
\begin{theorem}\label{thm:surface:magnus:sub}
Let $M$ be a Magnus submonoid of  
\[
S_g = \Gpres{a_1, \ldots, a_g, b_1, \ldots, b_g}{[a_1,b_1]\ldots[a_g,b_g] = 1}
\]
Then membership in any finitely generated submonoid of $M$ is decidable.
Similarly, if $N$ is a 
Magnus submonoid of 
\[
\mathcal{N}_g = \Gpres{a_1, \ldots, a_g}{a_1^2 \ldots a_g^2 = 1}
\]
then membership in any finitely generated submonoid of $N$ is decidable.
\end{theorem}

\begin{proof}
Let $X = \{a_1, \ldots, a_g, b_1, \ldots, b_g\}^{\pm 1}$. Since $M = \Mgen{Y}$ is a Magnus submonoid of $S_g$ one has $Y \subsetneq X$. Pick an element $x \in X \setminus Y$. Now, let $N = \Mgen{Z}$, be a finitely generated submonoid of $M$. Every element $z$ of $Z$ has $x$-exponent sum $\geq 0$, as of product of such elements. Now the result follows from \Cref{thm: exponent_sum_one_relator2:NewCorrected} by arguing in the same way as in the proof of \cref{thm:MagnusSubmonoidsSurfaceGroupsGeneral}. The proof for $\mathcal{N}_g$ follows analogously. 
\end{proof}

\subsection{Baumslag--Solitar groups}

In this subsection, we show that membership in Magnus submonoids is decidable in all Baumslag--Solitar groups. We remark that the general submonoid membership problem for Baumslag--Solitar groups remains open in general, although it was proved in \cite{Cadilhac2020rational} that it is decidable in $BS(1,n)$.

\begin{theorem}\label{thm: magnus in BS groups}
For all $m, n \in \mathbb{Z}$, membership in Magnus submonoids of the Baumslag--Solitar group 
\[
BS(m,n) = \langle a, t \mid t a^m t^{-1} = a^n \rangle
\]
is decidable.
\end{theorem}

\begin{remark} 
One might wonder whether for any $2$-generator one-relator group we can show that membership in Magnus submonoids is decidable. This question is likely to be hard since if it is decidable then by Theorem~\ref{AdianPositiveSubmonoid} it would resolve the currently open word problem for $2$-generator one-relator inverse monoids with cyclically reduced relator $uv^{-1}$ where $u$ and $v$ are both positive words.  
\end{remark}

For the proof of \cref{thm: magnus in BS groups}, we begin with the case $m, n \geq 1$, as established in the following lemma.

\begin{lemma}\label{lem: magnus in positive BS groups}
For all $1 \leq m < n$, membership in Magnus submonoids of Baumslag--Solitar groups $BS(m,n)$ is decidable.
\end{lemma}

To prove the lemma, we use the following result:
\begin{prop}[{\cite[Section 2]{diekert2011computing}}]\label{diekert_rewriting_BS}
Let $1 \leq m < n$, and consider the Baumslag--Solitar group
\[
BS(m, n) = \langle a, t \mid t a^m t^{-1} = a^n \rangle.
\]
Set $A = a^{-1}$, $T = t^{-1}$.
The following finite rewriting system is complete:
\begin{align*}
a A &\to 1, & A a &\to 1, & 
t T &\to 1, & T t &\to 1, \\
a^n t &\to t a^m, & A t &\to a^{n-1} t A^m, & 
a^m T &\to T a^n, & A T &\to a^{m-1} T A^n.
\end{align*}
\end{prop}
We refer the reader to \cite[Chapter~12]{Holt2005} for the definition of complete rewriting system. 
It follows from \cref{diekert_rewriting_BS} that the irreducible words with respect to that rewriting system give a set of normal forms for the elements of $BS(m,n)$ where a word is irreducible if none of the rewrite rules can be applied to it. When working with this complete rewriting system given any word $w$ over the generators we use $\overline{w}$ to denote the unique irreducible word such that $w = \overline{w}$ in the group.      

\begin{proof}[Proof of \Cref{lem: magnus in positive BS groups}]

Given a nonempty proper subset $S \subsetneq \{a, A, t, T\}$, we want to decide whether an element $g \in BS(m,n)$ belongs to the Magnus submonoid $\Mgen{S}$.

Note that for $S = \emptyset$, we have $\Mgen{S} = \{1\}$, and the problem reduces to the word problem in $BS(m,n)$, which is decidable since $BS(m,n)$ is an HNN extension $\mathbb{Z} *_{\mathbb{Z}}$.

Next, for $\emptyset \neq S \subsetneq \{a, A, t, T\}$ we consider the following cases:

\smallskip
\noindent     (1) If $S \subseteq \{a, A\}$ or $S \subseteq \{t, T\}$, then by Magnus' Freiheitssatz, 
   and Magnus' Theorem (see \cite[Chapter 2, Section 5]{lyndon1977combinatorial}),  
    the subgroup $\langle S \rangle$ is free in $BS(m,n)$ and membership in the Magnus subgroup $\langle S \rangle$ is decidable. Hence membership in the submonoid $\Mgen{S}$ (which lies inside the free group $\langle S \rangle$) is decidable by Benois' theorem \cite{Benois}.

\smallskip
\noindent     (2) If $\{t\} \subseteq S \subseteq \{a, A, t\}$ or $\{T\} \subseteq S \subseteq \{a, A, T\}$, the result follows from Lemma~\ref{lem:DG}  
    by taking $W_0 = S \cap \{a, A\}$ and $W_1 = \{1\}$.

\smallskip
\noindent     (3) The remaining cases to consider are
    \[
    S_1 = \{a, t, T\} \quad \text{and} \quad S_2 = \{A, t, T\}.
    \]
\noindent   (a) 
        Let $w \in S_1^\ast = \{a, t, T\}^\ast$. Since the only rewrite rules in the complete rewriting system in the statement of Proposition~\ref{diekert_rewriting_BS} with left-hand side belonging to $S_1^\ast$ are 
        \[
        t T \to 1, \quad T t \to 1, \quad a^n t \to t a^m, \quad  a^m T  \to T a^n
        \]
       it follows that the reduced form $\overline{w}$ also belongs to $S_1^\ast$. 
Given any word $u \in \{a, A, t, T\}^\ast$, if $u \in \Mgen{S_1}$ this means that there exists $w \in S_1^\ast$ such that $u=w$ in $BS(m,n)$. Since the system is complete it follows that $\overline{u} \equiv \overline{w} \in S_1^*$.   
Conversely, given any word $u \in \{a, A, t, T\}^\ast$, if $\overline{u} \in \{a, t, T\}^\ast = S_1^*$ then $u \in \Mgen{S_1}$. 
Hence, $u \in \Mgen{S_1}$ if and only if $\overline{u} \in S_1^\ast$.

\smallskip
\noindent (b) For $S_2 = \{A, t, T\}$, we observe that it is straightforward to verify that the rewriting system
 \begin{align*}
A a &\to 1, & a A &\to 1, &
t T &\to 1, & T t &\to 1, \\
A^n t &\to t A^m, & a t &\to A^{n-1} t a^m, &
A^m T &\to T A^n, & a T &\to A^{m-1} T a^n.
\end{align*}       
is a (different) finite complete rewriting system defining the same group        
\[
BS(m,n) = \langle a, t \mid t a^m t^{-1} = a^n \rangle
\]
Now working with respect to this complete rewriting system we can prove that membership in $\Mgen{S_2} \leq BS(m,n)$ is decidable using exactly the same argument as in the previous case (a).  
\end{proof}

Now we show that membership in Magnus submonoids of any Baumslag--Solitar groups $BS(m,n)$ is decidable.

\begin{proof}[Proof of \Cref{thm: magnus in BS groups}]

If one of $m, n$ is equal to $0$, or if $m=n$, then the group $BS(m, n)$ is virtually a direct product of $\Z$ with a finite-rank free group, hence it has decidable 
submonoid membership problem 
\cite[Theorem 1]{lohrey2008submonoid}, and hence membership in Magnus submonoids is decidable as well. 

If $1 \leq m < n$, then \Cref{lem: magnus in positive BS groups} gives the result. 
The case $1 \leq n < m$ follows, as $BS(m, n) \cong BS(n, m)$.
The case $m, n < 0$ follows from above, because $BS(-m, -n) = BS(n, m)$. 
The last case to consider is $m < 0$, and $n > 0$. The case $m > 0, n < 0$ is analogous. 
Note that points (1) and (2) in the proof of \Cref{lem: magnus in positive BS groups} above do not use rewriting systems, so we can apply the same proof here as well. On the other hand, for     
	\[
    S_1 = \{a, t, T\} \quad \text{and} \quad S_2 = \{A, t, T\}.
    \]
one can easily verify that $\Mgen{S_1} = \Mgen{S_2} = BS(m, n)$ and membership in these submonoids is clearly decidable. 
\end{proof}

\subsection{Two-generator free-by-cyclic one-relator groups}\label{Two-generator free-by-cyclic one-relator groups}

We follow the explanation of Brown's criterion given in \cite[Section~5]{dunfield2006random}.
Let 
$
G = \langle a, b \mid w = 1 \rangle
$
be a two-generator free-by-cyclic group. By definition, this means there is a surjective homomorphism 
$
\phi: G \to \mathbb{Z}
$
such that $\ker(\phi)$ is a free group.

It follows from work of Moldavanskii \cite{Moldavanskii1967} and Brown \cite{Brown1987}
that in this situation, the map $\phi$ can be chosen so that $\ker(\phi)$ is finitely generated free. Moreover, given such a homomorphism $\phi$, the fact that the kernel is finitely generated is equivalent to a certain max-min condition. 
We now state this result as formulated in \cite[Theorem 5.1]{dunfield2006random}.

\begin{theorem} \cite[Theorem 4.3]{Brown1987} \label{thm:Brown}
Let 
$
G = \langle a, b \mid w = 1 \rangle
$
be a two-generator free-by-cyclic group. Let $w_1, \ldots, w_n \equiv w$ denote the set of all prefixes of the word $w$.
 Then there exists a surjective homomorphism $\phi: G \to \mathbb{Z}$ such that $\ker(\phi)$ is a finitely generated free group, and one of the following holds:
\begin{enumerate}
    \item $\phi(a)$ and $\phi(b)$ are both nonzero, and the sequence $\phi(w_1), \ldots, \phi(w_n)$ has a unique minimum and a unique maximum;
    \item Exactly one of $\phi(a)$ or $\phi(b)$ is zero (and thus the other is nonzero since $\phi$ is surjective), and the sequence $\phi(w_1), \ldots, \phi(w_n)$ has exactly two maxima and exactly two minima, and $w$ is not equal to $a^2$ or $b^2$.
\end{enumerate}
\end{theorem}

Combining the result above with \cref{thm: exponent_sum_one_relator2:NewCorrected}, we obtain the following.

\begin{theorem}\label{thm:free:by:cyclic}
Let $G = \Gpres{a,b}{r = 1}$ be a $2$-generator free-by-cyclic one-relator group. Further, suppose that the relator $r$ has $a$-exponent sum equal to zero and $b$-exponent sum not equal to zero.
Then for any finite set of words $w_1, \ldots, w_k$ all with $a$-exponent sum at least zero, the membership problem in the submonoid 
\[
M = \Mgen{w_1, \ldots, w_k} \leqslant G
\]
is decidable.
In particular, $G$ has decidable positivity problem.
\end{theorem}
\begin{proof}
It follows from these assumptions that the homomorphism $\phi: G \to \mathbb{Z}$ is essentially unique,
in the sense that $\phi(b)=0$ and $\phi(a) \in \{1,-1\}$. Then it follows from 
Theorem~\ref{thm:Brown} (2) that the max-min condition in the statement of \cref{thm: exponent_sum_one_relator2:NewCorrected} is satisfied (this is explained in detail in the two paragraphs following \cite[Theorem~5.1]{dunfield2006random}). 
Hence, the result follows by applying \cref{thm: exponent_sum_one_relator2:NewCorrected}. 
\end{proof}

\begin{remark}
It is interesting to contrast the above theorem with the result  
\cite[Theorem 2.4]{margolis2005distortion}
which gives a condition on 2-generator one-relator groups under which membership in  positively and finitely generated submonoids is decidable; meaning the generators for the submonoid are positive words. 
\end{remark}

\begin{remark}\label{rmk:maxmin} 
It may be shown that there do exist examples 
of $2$-generator free-by-cyclic one-relator groups 
$G = \Gpres{a,b}{r = 1}$, where both the $a$- and $b$-exponent sums are zero, but
carrying out the Magnus rewriting procedure 
with respect to each of $a$ or $b$ does not yield a word satisfying the min-max condition from the statement of 
\cref{thm: exponent_sum_one_relator2:NewCorrected}. 
This occurs because, in such cases, there are many possible homomorphisms onto $\mathbb{Z}$.
What makes the case considered in the previous theorem easier to handle is that the conditions imply the homomorphism to $\mathbb{Z}$ is essentially unique.

In the case where both the $a$- and $b$-exponent sums are zero
Brown's theory does imply that the group admits \emph{some} one-relator presentation $\Gpres{c,d}{w=1}$ with $c$-exponent sum zero such that rewriting with respect to $c$ satisfies the max-min condition from the statement of \cref{thm: exponent_sum_one_relator2:NewCorrected}. This allows us to decide the positivity problem in this new presentation. However, typically this does not provide insight into the positivity problem for the original presentation.            
\end{remark}

It is currently unknown whether all finitely generated free-by-cyclic $2$-generator one-relator groups have decidable positivity problem. 
It is worth noting, as explained in the introduction, that there do exist finitely generated free-by-cyclic $2$-generator one-relator groups with undecidable submonoid membership problem; for example, the group
$
G = \Gpres{a,t}{t(ata^{-1}) = (ata^{-1})t}.
$
However, we can show the following.

\begin{prop}\label{prop:BurnsGroup} 
Membership in Magnus submonoids of 
\[
G = \Gpres{a,t}{t(ata^{-1}) = (ata^{-1})t}
\]
is decidable.
\end{prop}

\begin{proof} 
Performing Magnus rewriting with respect to the letter $t$ on the defining relator $r$ yields the word  
$\overline{r} = a_1 a_2^{-1} a_1 a_0^{-1}$,
which satisfies the max-min condition of \cref{thm: exponent_sum_one_relator2:NewCorrected}.
It then follows from that result that membership in any Magnus submonoid generated by any subset $S$ of $\{a,a^{-1},t,t^{-1}\}$ that does not contain both $t$ and $t^{-1}$ is decidable. The only remaining Magnus submonoids to consider are
$\Mgen{t, t^{-1}, a}$ and $\Mgen{t, t^{-1}, a^{-1}}$.  
We give the proof only for the first case $M = \Mgen{t, t^{-1}, a}$, as the second one goes analogously. 
As a free-by-cyclic group over the free group $F_{ \{ a_0,a_1 \}}$, $G$ has the presentation
\[
G = \Gpres{a_0, a_1, t}{t a_0 t^{-1} = a_1, \quad t a_1 t^{-1} = a_1 a_0^{-1} a_1}.
\]
Working backwards by substituting the redundant generator $a_1$, one recovers the original presentation with $t = t$ and $a = a_0$. Deciding membership in 
\[
M = \Mgen{t, t^{-1}, a} = \Mgen{t, t^{-1}, a_0} \leq G
\]
reduces to deciding membership in the infinitely generated submonoid 
\[
N = \Mgen{t^i a_0 t^{-i} \mid i \in \Z} \leq F_{ \{ a_0, a_1 \}},
\]
which is a submonoid of the free group $F_{ \{ a_0,a_1 \}}$. 
Indeed, any word $\alpha$ over $\{t, t^{-1}, a_0\}$ can be written in free-by-cyclic normal form as $\alpha = t^j \beta$,  
where $\beta \in F_{ \{ a_0, a_1 \}}$. Since $t$ and $t^{-1}$ are in the generating set, $\alpha$ belongs to $M$ if and only if $\beta$ belongs to $M$, which in turn holds if and only if $\beta \in N$. This is because any word over $\{t, t^{-1}, a_0\}$ with $t$-exponent sum zero can be expressed as a product of elements from $N$.

Let $\theta \in \Aut(F_{ \{ a_0,a_1 \} }$ denote the automorphism given by conjugation by $t$, i.e. $\theta(g) = t g t^{-1}$. Then
\[
N = \Mgen{\theta^i(a_0) \mid i \in \Z} = \Mgen{t^i a_0 t^{-i} \mid i \in \Z}.
\]

Recall:
\[
\theta(a_0) = a_1, \quad \theta(a_1) = a_1 a_0^{-1} a_1, \quad
\theta^{-1}(a_0) = a_0 a_1^{-1} a_0, \quad \theta^{-1}(a_1) = a_0.
\]
Hence,
\[
N = \Mgen{\ldots, \,
(a_0 a_1^{-1})^2 a_0, \,
(a_0 a_1^{-1}) a_0, \,
a_0, \,
a_1, \,
(a_1 a_0^{-1}) a_1, \,
(a_1 a_0^{-1})^2 a_1, \,
\ldots}.
\]
Denote $X = \{ \theta^i(a_0) : i \in \Z \}$. 
We may observe that $\theta^{i+1}(a_{0})$ is strictly longer than $\theta^i(a_0)$ for all $i \geq 1$, and also 
$\theta^{-(i+1)}(a_{0})$ is strictly longer than $\theta^{-i}(a_0)$ for all $i \geq 0$.
Moreover, in this example, no pair of elements of $X$ admit cancellation when multiplied together since the first and last letters of every word in $X$ are all positive letters. 
This means that every word in $X$ is a reduced word in the free group, and any finite concatenation of words from $X$ is also a reduced word. 
Thus, to decide membership of any $\gamma \in F_{ \{ a_0,a_1 \}}$ in $N$, one only needs to consider all products of elements of $X$ with total length up to the length of $\gamma$ as a reduced word, which is finite and computable. Hence membership in $N$, and therefore in $M$, is decidable.
This completes the proof.
\end{proof}

Note that in the previous proof the defining relator also has $a$-exponent sum zero; however, performing Magnus rewriting with respect to $a$ produces the word 
$
t_0 t_1 t_0^{-1} t_1^{-1},
$
which does not satisfy the max-min condition of \cref{thm: exponent_sum_one_relator2:NewCorrected}.

Looking at the proof of \cref{prop:BurnsGroup}, we observe that the decidability of the membership problem in Magnus submonoids of $2$-generator free-by-cyclic one-relator groups is closely connected to the following natural problem in the theory of free groups and automorphisms:

\begin{question}
\label{qn:orbit}
Let $F = F_{\{ a_0,a_1 \} }$ be the free group on two generators, and let $\theta: F \to F$ be an automorphism. Define the submonoid
\[
M = \Mgen{\theta^i(a_0) \mid i \in \Z}.
\]
Is there an algorithm that, given any word $w \in F$, decides whether $w \in M$?
\end{question}

\begin{remark}\label{rmk:FreeByCyclic}
A solution to this ``shifted membership problem" is likely to be essential for resolving the Magnus submonoid membership problem in $2$-generator one-relator free-by-cyclic groups.
Furthermore, since they are free-by-cyclic, similar questions arise in the context of surface groups. In that setting, resolving the submonoid membership problem often reduces to questions of the following kind: given a finite subset $Z$ of a free group and an automorphism $\theta$, determine membership in the submonoid
$
M = \Mgen{\theta^i(Z) \mid i \in \Z}.
$
where 
$\{  \theta^i(Z) \mid i \in \Z \} = 
\bigcup_{i \in \Z} \theta^i(Z)$. 
\end{remark}

It was proved by Brinkmann in \cite{Brinkmann2010} 
that if $\phi$ is an automorphism of a free group $F_n$, then there exists an algorithm that, given any pair $u, v \in F_n$, decides whether there exists $n \in \Z$ such that $\phi^n(u) = v$. More recently, this result was extended to arbitrary (not necessarily injective) endomorphisms by Logan \cite{Logan2022}.
However, the following result shows that while membership in the \emph{set} $\{\phi^i(X) \mid i \in \Z\} = \bigcup_{i \in \mathbb{Z}} \phi^i(X)$, where $X$ is a finite set, is decidable by Brinkmann’s theorem, membership in the submonoid $\Mgen{\phi^i(X) \mid i \in \Z}$ is, in general, undecidable.
The question of whether membership in the subgroup $\Ggen{\phi^i(X) \mid i \in \Z}$ is decidable remains open and is closely related to the open subgroup membership problem for free-by-cyclic groups.

\begin{theorem}\label{thm:AutFreeGroupMonGen} 
There is a finite subset $X \subseteq F_2$ of the free group $F_2$ of rank $2$ an a single automorphism $\phi \in \Aut(F_2)$ such that membership in the submonoid 
$
\Mgen{\phi^i(X) \; \; (i \in \Z) }
$   
is undecidable.  
  \end{theorem}

At the same time we shall prove the following result. 

\begin{theorem}\label{thm:AutFreeGroupMonGenFreeByCyclic} 
There is a free by cyclic group 
\[
G = F_2 \rtimes_\phi \Z = 
\Gpres{a,b,t}{tat^{-1} = \alpha, tbt^{-1} = \beta}
\]
and a finite subset $X \subseteq \{ a,a^{-1}, b, b^{-1} \}^*$ such that membership in $\Mgen{t,t^{-1}, X} \leq G$ is undecidable.    
  \end{theorem}
Note this contrasts with the result 
\cref{lem:DG}
showing that under the same conditions membership in both $\Mgen{t, X}$ and $\Mgen{t^{-1}, X}$ will be decidable.    

\begin{proof}[Proof of \cref{thm:AutFreeGroupMonGen} and \cref{thm:AutFreeGroupMonGenFreeByCyclic}]
Let 
\[
G = \Gpres{a,t}{t(ata^{-1}) = (ata^{-1})t}.
\] 
Let 
\[A(S) = \langle \; t_i \; (0 \leq i \leq 5) \; \mid \; t_j t_{j+1} = t_{j+1} t_j \; (0 \leq j \leq 4) \; \rangle \] 
be the right-angled Artin group $A(S)$ of a line $S$ with six vertices $\{ t_0, t_1, t_2, t_3, t_4, t_5 \}$, with $t_i$ adjacent to $t_{i+1}$ for all $i$.
Then by \cite[Lemma 3.2]{vogtmann2013gl} the subgroup $H = \langle t_0, t_1, t_2, t_3, t_4 \rangle \leq A(S)$ is isomorphic to the right-angled Artin group of the path with five vertices, and likewise so is the subgroup $Q = \langle t_1, t_2, t_3, t_4, t_5 \rangle \leq A(S)$. Hence we can form the HNN-extension of $A(S)$ with respect to the isomorphism $\phi:H \rightarrow Q$ where $\phi(t_i) \mapsto t_{i+1}$ for $0 \leq i \leq 4$. This group is defined by the presentation 
\[ \langle t_0, t_1, t_2, t_3, t_4, t_5, a \mid t_i t_{i+1} = t_{i+1} t_i, \quad a t_i a^{-1} = t_{i+1} \; (0 \leq i \leq 4) \rangle. \]
Preforming Tietze transformations to eliminate the redundant generators $\{ t_1, t_2, t_3, t_4, t_5 \}$ this presentation becomes   
\[ \langle a, t_0 \mid t_0 (at_0a^{-1}) = (at_0a^{-1}) t_0  \rangle \]
which is clearly isomorphic to $G$. It follows that the subgroup of 
$G$ generated by the set $\{ t, a^2 t a^{-2}, a^3 t a^{-3}, a^4 t a^{-4}, a^5 t a^{-5} \}$ is isomorphic to 
the subgroup of $A(S)$ generated by the set $\{t_0, t_2, t_3, t_4, t_5 \}$ which, by \cite[Lemma 3.2]{vogtmann2013gl} is isomorphic to the free product $\Z \ast A(P_4)$ where $t_0$ generates $\Z$ and $\{ t_2, t_3, t_4, t_5 \}$ generates $A(P_4)$ with relations $t_i t_{i+1} = t_{i+1} t_i$ for $2 \leq i \leq 4$.        
Hence the subgroup of $G$ generated by $\{ t, a^2 t a^{-2}, a^3 t a^{-3}, a^4 t a^{-4}, a^5 t a^{-5} \}$ is isomorphic to the free product $\Z \ast A(P_4)$ where $\Z$ is generated by $t$ and $A(P_4)$ is generated by $\{ a^2 t a^{-2}, a^3 t a^{-3}, a^4 t a^{-4}, a^5 t a^{-5} \}$ where $a^i t a^{-i}$ commutes with $a^{i+1} t a^{-{(i+1)}}$ for all $i$.

Now let 
\[
K \leq \Ggen{a^2 t a^{-2}, a^3 t a^{-3}, a^4 t a^{-4}, a^5 t a^{-5}} \cong A(P_4)
\] 
be a finitely generated submonoid in which membership is undecidable. Such a submonoid $K$ exists by \cite[Theorem~7]{lohrey2008submonoid}. Define 
\[
M = \Mgen{t, t^{-1}, K} \leq G. 
\]
We claim that $M \leq G$ is a finitely generated submonoid of $G$ in which membership is undecidable.
To see this first note that $G$ has decidable subgroup membership problem since $G$ virtually embeds into $A(P_4)$ by \cite[proof of Theorem~1.2]{niblo2001}, $A(P_4)$ has decidable subgroup membership problem by e.g. \cite{kapovich2005foldings}, and having decidable subgroup membership problem is inherited by finite index extensions; see e.g. \cite{lohrey2008submonoid}. It follows that if we set  
\[
L = \Ggen{a^2 t a^{-2}, a^3 t a^{-3}, a^4 t a^{-4}, a^5 t a^{-5}} \cong A(P_4) \leq G
\]   
then membership in the subgroup $L \leq G$ is decidable.  
Working with free product normal forms for elements in the group 
\[
\Z \ast A(P_4)  = \Z \ast L \cong 
  \Ggen{t,L} =
\langle  t, a^2 t a^{-2}, a^3 t a^{-3}, a^4 t a^{-4}, a^5 t a^{-5} \rangle 
\leq G 
\]
it is straightforward to verify that 
$M \cap L = K$.
Hence since membership in $L$ is decidable but in $K$ is undecidable, membership in $M \leq G$ must be undecidable. Performing Magnus rewriting with respect to the letter $t$ on $G$  to the defining relator
$r = t(ata^{-1})((ata^{-1})t)^{-1}$ 
gives the word $\overline{r} = a_1 a_2^{-1} a_1 a_0^{-1}$.
As a free-by-cyclic group,  over the free group $F_{ \{ a_0,a_1 \}}$, the group 
$G$ has the form 
\[
G = \Gpres{a_0, a_1, t}{ta_0t^{-1} = a_1, ta_1 t^{-1} = a_1 a_0^{-1} a_1}
\]
Note that every element of $G$ can be written uniquely in the form $u t^j$ for some $u \in F_{\{ a_0, a_1 \}}$ and $i \in \Z$. Let $\{ u_1 t^{j_1}, \ldots, u_k t^{j_k} \}$ be a finite generating set for the submonoid $K \leq G$. Then we have 
\begin{align*} 
M 
&= \Mgen{t, t^{-1}, K} 
= \Mgen{t, t^{-1}, u_1 t^{j_1}, \ldots, u_k t^{j_k} } 
= \Mgen{t, t^{-1}, u_1, \ldots, u_k } 
  \end{align*}
where $u_1, \ldots, u_k \in F_{\{ a_0,a_1 \}}$. Taking $X = \{u_1, \ldots, u_k \} \subseteq F_{\{ a_0,a_1 \}}$ this completes the proof of  \cref{thm:AutFreeGroupMonGenFreeByCyclic}. 

For \cref{thm:AutFreeGroupMonGen} we take the same set $X = \{u_1, \ldots, u_k \} \subseteq F_{\{ a_0,a_1 \}}$ and let $\phi$ be the automorphism sending $a_0 \mapsto a_1$ and $a_1 \mapsto a_1 a_0^{-1} a_1$. Then for any element $\alpha \in F(a_0,a_1)$ and $i \in \Z$ we have 
\begin{align*} 
\alpha t^i \in \Mgen{t, t^{-1}, u_1, \ldots, u_k }   
& \Leftrightarrow 
\alpha \in \Mgen{t^i u_l t^{-i} \; \; (i \in \Z, 1 \leq l \leq k) } \\
& \Leftrightarrow 
\alpha \in \Mgen{\phi^i(X) \; \; (i \in \Z) }. 
\end{align*}
Since membership in $\Mgen{t, t^{-1}, u_1, \ldots, u_k }$ is undecidable it follows that membership in the submonoid $\Mgen{\phi^i(X) \; \; (i \in \Z) }$ is undecidable. This completes the proof of \cref{thm:AutFreeGroupMonGen}. \end{proof}

\begin{remark} The surface group $S_2$ is isomorphic to 
\[
H = \Gpres{t,b_0,c_i,d_i \; (i \in \Z)}{t b_0 t^{-1} = (d_0c_0d_0^{-1}c_0^{-1})b_0, \ tc_i t^{-1} = c_{i+1}, \ td_i t^{-1} = d_{i+1}  }
\]
This gives an explicit description of $S_2$ as an (infinite rank) free by cyclic group, with respect to the automorphism $\theta$ of the infinite rank free group $F$ with basis $\{ b_0,c_i,d_i \; (i \in \Z) \} $ where $\theta$ maps $b_0 \mapsto (d_0c_0d_0^{-1}c_0^{-1})b_0$, $c_i \mapsto c_{i+1}$ and $d_i \mapsto d_{i+1}$. A key step towards solving the submonoid membership problem in $S_2$ is to first show membership can be decided in submonoids generated by sets of the form $\{t, t^{-1}, u_1, \ldots, u_k \}$ where $u_1, \ldots u_k$ are all elements of the infinite rank free group $F$. \end{remark}

\section{Hyperbolic groups with undecidable positivity problem}\label{sec:hyperbolic}
In the previous section we have seen how our results may be applied to show that for various families of one-relator groups the membership problem in Magnus submonoids is decidable. In particular we have seen that for certain one-relator groups we can decide membership in the positive submonoid, although the question remains open for one-relator groups in general. 
As mentioned in the introduction, it follows from classical work of Boone that there are finitely presented groups for which the problem of deciding membership in the positive submonoid,   
called the quasi-Magnus problem, is undecidable. 
In \cite{McCammondMeakin} McCammond and Meakin asked whether this problem is decidable in hyperbolic groups. Specifically they ask:

\begin{quote} {\bf Question.} (McCammond and Meakin, 2006 \cite[Problem~1]{McCammondMeakin}). Can one decide whether an element in a hyperbolic group $G$ generated by $S$ is represented by a positive word in the generators? 
\end{quote}
In this section we shall prove the following result giving a negative answer to their question.

\begin{theorem}\label{thm:posUndecHyp}
There is a hyperbolic group~$G$ generated by a finite set~$S$ such that there is no algorithm that takes an element of~$G$ as a word over~$S \cup S^{-1}$ and decides whether it can be written as a word in the generators~$S$.       
\end{theorem}
We shall also prove the following slightly stronger result.  
\begin{theorem}\label{thm:ResFin} 
There is a residually finite hyperbolic group~$G$ generated by a finite set~$S$ such that there is no algorithm that takes an element of~$G$ as a word over~$S \cup S^{-1}$ and decides whether it can be written as a word in the generators~$S$.       
\end{theorem}
Note that it is a long standing open question of Gromov whether every hyperbolic group is residually finite.
The results above are all obtained as corollaries of the following general construction by combining it with the Rips construction (and a variant of it from the literature), and the fact that each of the classes defined above (e.g. hyperbolic groups, residually finite groups) is closed under taking free products. 

We shall need the following lemma:
\begin{lemma}\cite[Lemma~3.6]{gray2020undecidability}\label{lem:L36} 
Let $H$ be a group and let $W$ be a finite subset of $H$. Let $T$ be the submonoid of $H$ generated by $W$, and let $S$ be the submonoid of $H \ast \mathrm{FG}(t)$ generated by $\{t\} \cup H \cup tWt^{-1}$. Then for all $h \in H$ 
\[ tht^{-1 } \in  S \Leftrightarrow  h \in T.  \]
  \end{lemma} 

\begin{theorem}\label{thm:positivity:fails:in:general}
Let~$G$ be a finitely presented group and suppose that~$G$ contains a finitely generated submonoid in which membership is undecidable. Then there is a finite group presentation~$\Gpres{A}{R}$ for~$G \ast \Z$ such that there is no algorithm that takes a word~$w \in (A \cup A^{-1})^*$ as input and decides whether~$w$ can be written as a positive word, that is, whether there exists a word~$\pi_w \in A^*$ such that~$w=\pi_w$ in~$\langle A \mid R \rangle$.     
\end{theorem}
\begin{proof} 
Let~$G$ be a finitely presented group that contains a finitely generated submonoid~$M$ in which membership is undecidable. Let~$Y \subseteq G$ be a finite monoid generating set for the group $G$, and~$X \subseteq G$ a finite generating set for the submonoid~$M$. Take the free product~$K = G \ast \Z \cong G \ast F_{\{ t \} }$ where~$F_{\{ t \} }$ is the free group on~$t$. Set 
\[
A = Y \cup \{t\} \cup \{txt^{-1} : x \in X \}.
\]
Then $A$ is a group generating set for~$G \ast \Z$, since it contains~$Y \cup \{t\}$ which already is a group generating set for~$G \ast \Z$. 

We claim that there is no algorithm that takes an element of~$G \ast \Z$ as a word over~$A \cup A^{-1}$ and decides whether that element can be written as a positive word in the generators~$A$. 
In other words, we want to show that 
membership in the submonoid~$\Mgen{A} \leq K$ is undecidable. 

Applying Lemma~\ref{lem:L36} with $H=G$, $W = X$, $T = \Mgen{W} = \Mgen{X} = M$ and
\[
S = \Mgen{\{t\} \cup H \cup tWt^{-1}} = \Mgen{A}  
\]
it follows that for all $g \in G$ we have 
\[
tgt^{-1} \in \Mgen{A} \Leftrightarrow g \in M. 
\] 
Hence for any word $w \in F_A$ we have    
\[
twt^{-1} \in \Mgen{A} \Leftrightarrow w \in M. 
\]
But by assumption there is no algorithm that takes a word $w \in F_A$ and decides whether $w \in M$. Thus there is no algorithm that takes a word $w \in F_A$ and decides whether $twt^{-1} \in \Mgen{A}$. Hence membership in    
$\Mgen{A} \leq K$ is undecidable. 
This then completes the proof of the theorem by taking any finite presentation  
$\Gpres{A}{R}$ for~$G \ast \Z$ with respect to the generating set $A$, which is possible since  
$G \ast \Z$ is finitely presented. 
\end{proof}

We end this section by showing how Theorem~\ref{thm:positivity:fails:in:general} can be applied to establish the results above about hyperbolic groups. The key point is that the classes in question are preserved under taking a free product with $\Z$.

\begin{proof}[Proof of \cref{thm:posUndecHyp}]
By Rips' result \cite{rips1982subgroups} there is a finitely presented hyperbolic group $G\cong \Gpres{Y}{Q}$ containing a finitely generated submonoid (in fact a subgroup) in which membership is undecidable. Then the group $K = \Gpres{Y,t}{Q} \cong G \ast  \Z$ is a finitely presented hyperbolic group and by \cref{thm:positivity:fails:in:general} there is a finite presentation $K \cong \Gpres{A}{R}$ such that the positivity problem for $\Gpres{A}{R}$ is undecidable. The property of a group presentation being hyperbolic group is independent of the choice of finite presentation for that group.
Hence $\Gpres{A}{R}$ is a finitely presented hyperbolic group with undecidable positivity problem.  
\end{proof}

\begin{proof}[Proof of \cref{thm:ResFin}]
By Wise's result \cite{wise2003residually} there is a residually finite version of Rips's construction. So there is a finitely presented residually finite hyperbolic group $G\cong \Gpres{Y}{Q}$ containing a finitely generated submonoid (in fact a subgroup) in which membership is undecidable. The rest of the proof follows the proof of \cref{thm:posUndecHyp} above.
\end{proof}

\section{Membership in Other Submonoids of Surface Groups}\label{sec:other}

As mentioned above, it remains an open problem whether the submonoid membership problem is decidable in general for surface groups. 
Several of the results proved above in previous sections
show that membership in certain submonoids is decidable. For instance, Theorem~\ref{thm:PositiveSubonoidsOfSurfaceGroups} establishes that membership is decidable in the submonoid of $S_g$ consisting of elements whose generators all have non-negative exponent sum with respect to a fixed generator of~$S_g$.

In this section we shall prove several new results which show how to decide membership in some other families of submonoids of surface groups not covered by the previous results in this paper.   
We also include a discussion about what we believe is the primary obstacle to a general solution. This difficulty connects back to our earlier discussion of certain decision problems in free-by-cyclic groups, particularly those discussed in Remark~\ref{rmk:FreeByCyclic}.

\subsection{Low Rank Submonoids}
It is a classical fact that subgroups of surface groups generated by sufficiently small sets must be free. 
Combining this with Benois' Theorem and the fact that surface groups have decidable subgroup membership problem gives the following result.
\begin{prop}\label{prop:LowRank} 
Let $X$ be a subset of  
\[
S_g = \Gpres{a_1, \ldots, a_g, b_1, \ldots, b_g}{[a_1,b_1]\ldots[a_g,b_g]= 1}.
\]
If $|X| < 2g$, then membership in $\Mgen{X} \leq S_g$ is decidable. Furthermore, let $Y$ be a subset of  
\[
\mathcal{N}_g = \Gpres{a_1, \ldots, a_g}{a_1^2 \ldots a_g^2= 1}.
\]
If $|Y| < g$, then membership in $\Mgen{Y} \leq \mathcal{N}_g$ is decidable.       
\end{prop}

\begin{proof}
A well-known fact about subgroups of surface groups is that any subgroup $H_1$ of $S_g$ generated by fewer than $2g$ elements, or any subgroup $H_2$ of $\mathcal{N}_g$ generated by fewer than $g$ elements is a free group.

Hence given a subset $X \subseteq S_g$ with $|X| < 2g$ it follows that $\Mgen{X} \leq \Ggen{X} \leq S_g$ where $\Ggen{X}$ is a free group of finite rank (not necessarily with basis $X$).     
Then given any word $w$ over the generators of $S_g$ by  
\cite{scott1978subgroups} there is an algorithm that decides whether $w$ represents an element in  
$\Ggen{X}$. If it does not, then clearly $w \not\in \Mgen{X}$. On the other hand, if $w$ does represent an element of $\Ggen{X}$, then since this is a free group it then follows from    
Benois' Theorem that we can decide whether or not $w$ represents an element of $\Mgen{X}$.  
The proof for $\mathcal{N}_g$ is analogous using the fact that any subgroup of 
$\mathcal{N}_g$ generated by fewer than $g$ elements is a free group.
\end{proof}

\begin{remark}[Low-rank subgroups of one-relator groups]\label{rmk:prank}
Important work of Louder and Wilton~\cite{louder2022negative} proves a freeness result for one-relator groups that 
generalises the result mentioned in the previous proof that membership in low-rank subgroups of surface groups is decidable.  
Their result can also be applied to deduce decidability of the submonoid membership problem in low-rank submonoids of other one-relator groups.

Specifically, let $F$ be a free group, and let $w \in F$ be a non-primitive element (meaning that it does not belong to any basis of $F$). The \emph{primitivity rank} of $w$, denoted $\pi(w)$, is the smallest rank of a subgroup of $F$ in which $w$ is non-primitive. Louder and Wilton show that any subgroup of the one-relator group $G = F / \langle\langle w \rangle\rangle$ generated by fewer than $\pi(w)$ elements is free.
In this situation, provided that $G = F / \langle\langle w \rangle\rangle$ has decidable subgroup membership problem (just membership in free subgroups would be enough), membership in submonoids generated by fewer than $\pi(w)$ elements is decidable.

We note, however, that the general subgroup membership problem for one-relator groups remains open but is known to be decidable for one-relator groups with sufficiently high torsion; see e.g. \cite{mccammond2005coherence}. 

Note that in the case of the surface groups, one has $\pi([a_1,b_1]\ldots[a_g,b_g]) = 2g$, and $\pi(a_1^2 \ldots a_g^2) = g$ respectively \cite{puder2014primitive}, agreeing with the claim of Proposition \ref{prop:LowRank}.
\end{remark}

\subsection{Distortion functions}\label{sec: Distortions}

In the remainder of this section we shall develop and extend some tools from \cite{margolis2005distortion} to show how to decide membership in several other families of finitely generated submonoids of surface groups, and other related groups. We begin by recalling some definitions and results from \cite{margolis2005distortion} about distortion functions.

\begin{mydef}
Let~$M = \Mgen{S}$ be a submonoid of a monoid~$G = \Mgen{X}$, where $S$ and $X$ are finite. 
A \emph{distortion function} for~$M$ in~$G$ with respect to~$S$ and~$X$ is any non-decreasing function~$\delta \colon \N \rightarrow \N$ such that
$
|g|_S \leq \delta(|g|_X)
$
for all~$g \in M$.
\end{mydef}

We say that the distortion is \emph{linear} if there are constants $a,b>0$ such that $\delta(n) \leq an+b$ for all $n \in \N$. It follows easily from the definitions that the distortion being linear is independent of the choice of finite generating sets for $M$ and $G$. We say that a finitely generated submonoid $M$ of a finitely generated monoid $G$ is \emph{undistorted} if $M$ has linear distortion in $G$.

The following result follows from 
\cite[Proposition 1.1]{margolis2005distortion}.

\begin{prop}\label{prop: membership problem via distortion}
Let~$M = \Mgen{S}$ be a submonoid of the monoid~$G = \Mgen{X}$, where $S$ and $X$ are finite. 
Suppose that  $G$ has decidable word problem. 
Then the membership problem for~$M$ in~$G$ is decidable if and only if there exists a recursive distortion function for~$M$ in~$G$ with respect to~$S$ and~$X$.
\end{prop}

\begin{mydef}
A monoid~$M$ is \emph{graded} if it admits a finite generating set~$S$ such that every element of~$M$ can be expressed as a word over~$S$ in only finitely many ways.
\end{mydef}
Examples of graded monoids include free monoids and Artin monoids.

\begin{remark}
A graded monoid~$\Mgen{S}$ cannot contain any nontrivial subgroups (in particular, it cannot have any idempotents other than~$1$). The identity element~$1$ can be written using the generators in~$S$ only as the empty word; in particular, $1 \notin S$.
\end{remark}

\begin{mydef}
Let~$M$ be a graded monoid with respect to the generating set~$S$. The function~$\lambda_S \colon M \rightarrow \N$ defined by
\[
\lambda_S(g) = \max\{\,k\,|\, g = s_1 \dots s_k \text{ for some } s_i \in S, \; i = 1, \dots, k\}
\]
is called the \emph{upper word length function} of~$M$ with respect to~$S$.
\end{mydef}

It may be shown that for any two finite generating sets $S_1$ and $S_2$ of a monoid $M$, provided $1 \not\in S_1$ and $1 \not\in S_2$, then $M$ is graded with repect to $S_1$ if and only if $M$ is graded with respect to $S_2$.
Similarly, the upper word length function is independent of the system of generators, as long as they do not contain the identity; see \cite[Proposition~1.6]{margolis2005distortion}.

\begin{mydef}
Let~$M = \Mgen{S}$ be a graded submonoid of~$G = \Mgen{X}$ with both~$S$ and~$X$ finite. An  \emph{upper distortion function} for~$M$ in~$G$ with respect to~$S$ and~$X$ is any non-decreasing function~$\lambda \colon \N \rightarrow \N$ such that:
\[
\lambda_S(g) \leq \lambda ( | g |_X )
\]
for all~$g \in M$.
\end{mydef}

\emph{The actual upper distortion function}~$\lambda_{S, X} \colon \N \rightarrow \N$ is defined by
\[
\lambda_{S, X}(n) = \max\{\lambda_S(g) | g \in B_X^G(n) \cap M\}
\]
where~$B_X^G(n)$ is the ball of radius~$n$ with respect to~$X$ in~$G$, consisting of the elements $g$ of~$G$ such that
$|g|_X \leq n$.

\begin{mydef}
A graded submonoid~$M = \Mgen{S}$ of~$G = \Mgen{X}$, with~$S$ and~$X$ finite, is \emph{recursively embedded} if the upper distortion function~$\lambda_{S, X}$ is recursive.
\end{mydef}

A submonoid can be graded without admitting a recursive upper distortion function; that is, the membership problem may be undecidable even for graded submonoids embedded in groups with decidable word problem; see the discussion between Definition 1.9 and Definition 1.10 in \cite{margolis2005distortion}.

The following two propositions 
enable us to deduce that certain graded monoids have recursive embeddings by looking at homomorphic images.

\begin{prop}\cite[Proposition~1.11]{margolis2005distortion}
Let~$M = \Mgen{S}$ and~$M' = \Mgen{S'}$ be monoids, with~$S$ and~$S'$ finite, and let~$\varphi \colon M \rightarrow M'$ be a homomorphism with~$\varphi(S) \subseteq S'$. If~$M'$ is graded with respect to~$S'$ then~$M$ is graded with respect to~$S$ and for~$g \in M$ one has~$\lambda_S(g) \leq \lambda_{S'}(g')$ where~$g' = \varphi(g)$.
\end{prop}

\begin{prop}\cite[Proposition~1.12]{margolis2005distortion}\label{prop: homomorphisms and gradings}
Let~$M = \Mgen{S}$ be a submonoid of~$G = \Mgen{X}$ and~$M' = \Mgen{S'}$ a submonoid of~$G' = \Mgen{X'}$ with~$S, \, X, \, S'$ and  $X'~$ finite. 
Suppose~$M'$ is graded with respect to~$S'$, and suppose that there is a homomorphism 
$\varphi \colon G \rightarrow G'$ such that $\varphi(S) \subseteq S'$. 
Then~$M$ is graded with respect to~$S$ and if~$\lambda' \colon \N \rightarrow \N$ is an upper distortion function for~$M'$ in~$G'$, then~$\lambda_{S, X}(n) \leq \lambda'(Cn)$ where~$C = \max\{|\varphi (x)|_{X'} \mid x \in X\}$, holds for the actual upper distortion function for~$M$ in~$G$.
\end{prop}

The following corollary, also from \cite{margolis2005distortion}
provides
conditions under which recursive upper distortion functions can be lifted from homomorphic images and used to solve membership problems in the original monoid.

\begin{cor}\cite[Corollary~1.13]{margolis2005distortion}\label{cor:distortion:membership}
Let~$M = \Mgen{S}$ be a submonoid of~$G = \Mgen{X}$ and~$M' = \Mgen{S'}$ a submonoid of~$G' = \Mgen{X'}$ with~$S, \, X, \, S'$ and  $X'~$ finite. 
Suppose~$M'$ is graded with respect to~$S'$, and suppose that there is a homomorphism 
$\varphi \colon G \rightarrow G'$ such that $\varphi(S) \subseteq S'$. 
Let $C = \max\{|\varphi (x)|_{X'} \mid x \in X\}$.
If~$M'$ is recursively embedded in~$G'$, then~$M$ is recursively embedded in~$G$ and~$\lambda \colon \N \rightarrow \N$ given by~$\lambda(n) = \lambda'(Cn)$ is a recursive upper distortion function for~$M$ in~$G$, as well as for any finitely generated submonoid of~$M$. Then membership is decidable in any such submonoid of~$G$.
\end{cor}

\subsection{Submonoid membership problem for surface groups: distortion and fully residually free groups}\label{sed: Submonoid membership problem for surface groups: distortion and fully residually free groups}

The authors of \cite{margolis2005distortion} solved the prefix membership problem for the surface group~$S_2$ by mapping to the Heisenberg group. They use this in \cite[Proposition 2.10]{margolis2005distortion} to prove that 
every orientable surface group $S_g$ has decidable prefix membership problem with quadratic upper distortion function $\lambda(n) = n + n^2/4$.  
In Theorem~\ref{thm:PMPDistortionSurface} below we shall improve this result by showing all surface groups (including the non-orientable case) have decidable prefix membership problem with linear upper distortion function.   
Surface groups are limit groups and so admit many homomorphisms to free groups. Our interest here is to see how such homomorphisms can be used to decide membership in certain submonoids of surface groups. In particular this approach will give examples that are not covered by the results from   Section~\ref{sec: Magnus submonoid membership problem in one-relator groups} above. 
We shall need the following lemma. 
\begin{lemma}\label{submonoids of free groups are un-distorted}
Finitely generated submonoids of free groups are undistorted, that is, they have linear distortion function. 
\end{lemma}

\begin{proof}
Let $M = \Mgen{S}$, with $S = \{s_1, \ldots, s_k\}$, be a finitely generated submonoid of the free group $F = F_A$.
Let $\tilde{A} = A \cup A^{-1}$, and 
set 
\begin{align*}
L & = \max\{|s|_{\tilde{A}} \mid s \in S\}, & 
L' & = \max\{|u|_S \mid u \in M, |u|_A \leq L + 1\}.
\end{align*}
For $w \in M$, write $w = v_1 \ldots v_n$, with $v_i \in S$ and some minimal $n$. For any~$0 \leq k \leq n$, denote by $w_k = v_1\ldots v_k$ the prefixes of $w$. Define
\[
\varphi_w \colon [0, n] \rightarrow Cay_A(F)
\]
with $\varphi_w(0) = 1$, and $\varphi_w(k) = w_k$ for $0 < k \leq n$; and ask that $\varphi_w$ maps any interval $[k-1, k]$ homeomorphically into a geodesic $[\varphi_w(k-1), \varphi_w(k)]$ -- which is the geodesic between $w_{k-1}$ and $w_k$, for all $0 < k \leq n$.

Suppose $0 \leq i < j \leq n$, and consider a geodesic polygon of the form 
\[
[w_i, w_{i+1}] \cup \ldots \cup [w_{j-1}, w_{j}] \cup [w_j, w_i].
\]
Denote the geodesic $[w_j, w_i]$ as 
\[
w_j = q_0 \xrightarrow{a_1} q_1 \xrightarrow{a_2} \dots \xrightarrow{a_r} q_r = w_i
\]
with $a_i \in A \cup A^{-1}$. Let $1 \leq l \leq r$. The Cayley graph $Cay_A(F)$ is a tree, so it is 0-hyperbolic.
There exists $i \leq t \leq j-1$ and $x \in [w_t, w_{t+1}]$ with $d_A(x, q_l) = 0$. Since $d_A(w_t, w_{t+1}) = |v_{t+1}|_A \leq L$, there exists $i \leq z_l \leq j$ with $d_A(q_l, w_{z_l}) \leq \frac{L}{2}$. Fix also $z_0 = j$ and $z_r = i$. One has:
\begin{align*}
d_A(w_{z_{l-1}}, w_{z_l}) & \leq d_A(w_{z_{l-1}}, q_{l-1}) + d_A(q_{l-1}, q_l) + d_A(q_l, w_{z_l}) \\
& \leq \frac{L}{2} + 1 + \frac{L}{2} = L + 1,
\end{align*}
which means that~$|w_{z_l}^{-1} w_{z_{l-1}}|_A \leq L+1$. Note that it follows from the definition of $w_k$ that $w_{z_l}^{-1} w_{z_{l-1}}$ is an element of $M$. Also, by definition of $L'$ we have:
\[
|z_l - z_{l-1}| = |w_{z_l}^{-1} w_{z_{l-1}}|_S \leq L',
\]
for $1 \leq l \leq r$. Now it follows that:
\[
|j - i| =|z_0 - z_r| \leq \sum_{l=1}^{r} |z_l - z_{l-1}| \leq L'r = L' d_A(w_j, w_i) = L' d_A(\varphi_w(j), \varphi_w(i)).
\]

Now given $x, y \in [0, n]$, the last inequality above gives:
\begin{align*}
|x - y| & \leq |\floor x - \floor y| + 1 \leq L' d_A(\varphi_w(\floor x ), \varphi_w(\floor y )) + 1\\
& \leq L' \left[
d_A(\varphi_w(\floor x ), \varphi_w(x)) +
d_A(\varphi_w(x), \varphi_w(y)) + 
d_A(\varphi_w(y), \varphi_w(\floor y ))
\right] + 1 \\
& \leq L'd_A(\varphi_w(x), \varphi_w(y)) + 2L'L + 1.
\end{align*}
This means that the non-decreasing function~$\delta \colon \N \rightarrow \N$, given by 
\[
n \mapsto L'n + 2L'L + 1,
\]
is a linear distortion function for the finitely generated monoid~$M = \Mgen{S}$, inside the free group $F = F_A$.
\end{proof}

\begin{lemma}\label{lemma: free submonoids are graded in free groups with linear distortion}
Let $F = F_X$ be a free group and let $M = \Mgen{S} \leq F$ be a finitely generated free submonoid, where $X$ and $S$ are finite and $1 \not\in S$. Then $M$ is graded with respect to $S$ and has a linear distortion function.
\end{lemma}

\begin{proof}
Since $M$ is a finitely generated free monoid and $1 \not\in S$ it follows that $M$ is graded with respect to $S$. 
The linear distortion follows from Lemma \ref{submonoids of free groups are un-distorted}.
\end{proof}

\begin{theorem}\label{thm:DistortionMapToFreeGroup}
  Let $G = \Gpres{X}{R}$ be a finitely presented group and let $f:G \rightarrow F$ be a homomorphism onto a finite rank free group $F$. Let $A$ be a finite subset of $G$ such that $1 \not\in f(A)$. If $\Mgen{f(A)}$ is a graded monoid then $\Mgen{A} \leq G$ is graded with a linear distortion function and membership in this submonoid is decidable.  
\end{theorem}

\begin{proof}
Since $1 \not\in f(A)$ and $\Mgen{f(A)}$ is a graded monoid it follows from \cite[Proposition 1.6]{margolis2005distortion} that $\Mgen{f(A)}$ is graded with respect to $f(A)$. 
Now the result follows form \Cref{cor:distortion:membership} and \Cref{submonoids of free groups are un-distorted}.
\end{proof}

Since free monoids are graded monoids we obtain:  

\begin{cor}\label{cor:DistortionMapToFreeGroup}
  Let $G = \Gpres{X}{R}$ be a finitely presented group and let $f:G \rightarrow F$ be a homomorphism onto a finite rank free group $F$. Let $A$ be a finite subset of $G$ such that $1 \not\in f(A)$. If $\Mgen{f(A)}$ is a free monoid then $\Mgen{A} \leq G$ is graded with a linear distortion function and membership in this submonoid is decidable. 
\end{cor}

The next result shows how the previous results can be applied to certain 
one-relator groups.

\begin{theorem}\label{thm:S2matchedAdapted}
Let
$
G = \Gpres{A, A'}{u = u'}
$
where the map $a \mapsto a'$ is a bijection between the disjoint sets $A$ and $A'$, and $u' \in F_{A'}$ is obtained from $u \in F_A$ by replacing each letter $a$ with $a'$. Let $X \subseteq F_A$ be such that $1 \notin X$. If $\Mgen{X} \leq F_A$ is graded, then the submonoid $\Mgen{X \cup X'} \leq G$ is graded with linear distortion function. 
In particular, membership in $\Mgen{X \cup X'} \leq G$ is decidable.
\end{theorem}
\begin{proof}
This result follows from \cref{thm:DistortionMapToFreeGroup} by taking the 
map $\rho \colon G \rightarrow F_A$, induced by $a, a' \mapsto a$ for all corresponding pairs $a \in A$, and $a' \in A'$. 
\end{proof}

Observe that 
\cref{thm:S2matchedAdapted} 
applies in particular to the group
\[
S_2 = \Gpres{a,b,c,d}{aba^{-1}b^{-1} = dcd^{-1}c^{-1}}
\]
giving the following result.  

\begin{cor}\label{thm:S2matchedNew}
Let 
   \[
S_2 = \Gpres{a,b,c,d}{aba^{-1}b^{-1} = dcd^{-1}c^{-1}}
   \] 
Let $W \subseteq F_{ \{ a,b \} }$ with $1 \not\in W$ and let $W' \subseteq FG_{ \{ c,d \}}$ obtained from $W$ via the map $a \mapsto d$, $b \mapsto c$. Let $M = \Mgen{W \cup W'} \leq S_2$. If $\Mgen{W} \leq FG_{\{ a,b \}}$ is graded then $M$ is graded with linear distortion. In particular, membership in $M \leq S_2$ is decidable.    
\end{cor}

The following result improves on \cite[Proposition 2.10]{margolis2005distortion} where they show that $n + n^2/4$ is an upper distortion function for the prefix monoid of $S_g$.
Our theorem also extends the result to the non-orientable surface groups $\mathcal{N}_g$.      
 
\begin{theorem}\label{thm:PMPDistortionSurface} The orientable surface groups 
\[
S_g = \Gpres{a_1, \ldots, a_g, b_1, \ldots, b_g}{[a_1,b_1]\ldots[a_g,b_g]= 1},
\]
and the non-orientable ones
\[
\mathcal{N}_g = \Gpres{a_1, \ldots, a_g}{a_1^2 \ldots a_g^2= 1},
\]
all have decidable prefix membership problem with linear upper distortion function.
\end{theorem}
\begin{proof}
Set
\[
G = S_2 = \Gpres{a, b, c, d}{[a, b][c, d]= 1}
\]
and let $\varphi_g \colon S_g \rightarrow G$ be the homomorphism induced by 
\[
a_1 \mapsto a, \quad 
b_1 \mapsto b, \quad 
a_g \mapsto c, \quad 
b_g \mapsto d, \quad 
\text{and } a_i, b_i \mapsto 1, \text{ for } 2 \leq i \leq g-1.
\]
The nontrivial prefixes of the relator $[a_1, b_1] \cdots [a_g, b_g]$ map to nontrivial prefixes of $[a, b][c, d]$ under $\varphi_g$; hence, by Proposition \ref{prop: homomorphisms and gradings}, it is sufficient to demonstrate that the prefix monoid $P$ of $G$ embeds recursively in $G$, and to construct a linear upper distortion function for $P$ in $G$.

The prefix monoid $P$ of $G$ is generated by 
\[
\{ 
a, ab, aba^{-1}, aba^{-1}b^{-1}, 
d, dc, dcd^{-1}, dcd^{-1}c^{-1}
\}.
\]
Consider the homomorphism~$\rho \colon S \rightarrow F(x, y)$ given by
\[
\rho \colon a, d \mapsto x; \quad b,c \mapsto x^{-1}yx
.
\]
The prefix submonoid maps to~$\Mgen{x, yx, y, yx^{-1}yx}$, which is equal to the rank 3 free submonoid~$\Mgen{x, y, yx^{-1}yx}$. 
Since there is no cancellation in the free group when multiplying these elements, it is immediate that these words freely generate a monoid. In particular, we see that~$\Mgen{x, y, yx^{-1}yx} \leq FG_{\{ x,y \}}$ is graded with linear upper distortion function. 
It then follows from \cref{thm:S2matchedNew} that the group $S_2$ has decidable prefix membership problem with respect to the defining relator $aba^{-1}b^{-1}cdc^{-1}d^{-1}$, and with linear upper distortion function. The identity map~$\N \rightarrow \N$ given by~$n \mapsto n$ is a linear upper distortion function for the prefix submonoid.

For the non-orientable case set $K \coloneqq \mathcal{N}_2 = \Gpres{c, d}{c^2 d^2= 1}$, and consider the map~$\varphi_g: \mathcal{N}_g \longrightarrow K$, given by: 
\[
a_1 \mapsto c,\; a_g \mapsto d, \; \text{ and } a_i \mapsto 1 \text{ for all } 1 < i < g.
\]
Nontrivial prefixes of~$a_1^2 \ldots a_g^2$ map to nontrivial prefixes of
$a^2 d^2$ under~$\varphi_g$; so, by \Cref{prop: homomorphisms and gradings}, it is sufficient to show that the prefix monoid~$P = \Mgen{c, c^2, c^2 d}$ of~$K$ embeds recursively in~$K$ and construct an upper distortion function for~$P$ in~$K$. 

Consider the map~$\sigma \colon K \longrightarrow \Z \coloneqq \Ggen{x}$, given by~$c\mapsto x,\; d \mapsto x^{-1}$. The monoid $P' = \Mgen{x}$, which is the image of~$P$ under~$\sigma$, embeds recursively in~$\Z$ and the linear map~$\lambda \colon \N \longrightarrow \N$ defined by~$n \mapsto n$ provides a linear upper distortion function with respect to~$S' = \{x\}$, and~$X' = \{x, x^{-1}\}$.

This completes that the prefix monoid of $\mathcal{N}_g$ has a linear upper distortion function, and then decidability of the prefix membership problem follows from \cref{prop: membership problem via distortion}.    
\end{proof}
Results like Theorem~\ref{thm:S2matchedAdapted} above show how homomorphisms to free groups can be used as a tool for deciding membership in certain graded submonoids and obtain upper bounds for the corresponding distortion functions.   

Surface groups, and in fact a larger class of groups - known as pinched word groups, are fully residually free.

\begin{mydef}
Let~$F$ be a free group of rank at least~$2$, and~$\langle z \rangle \leqslant F$ a cyclic subgroup closed under taking roots. The amalgamated free products~$F\ast_{\langle z \rangle} F$ are known as {\it pinched word groups}
\end{mydef}

\begin{mydef}
Let~$G = A\ast_C B$, and suppose~$z \in B$ centralizes~$C$. The {\it Dehn twist} associated to~$z$ is the automorphism~$\delta_z \colon G \rightarrow G$, defined via:
    \[
        \delta_z(a) = a \text{ for } a \in A, \text{ and } \delta_z(b) = zbz^{-1} \text{ for } b \in B.
    \]
\end{mydef}

\begin{prop}[Proposition 2.9 in \cite{wilton2005introduction}]
    A pinched word group~$G = F\ast_{\langle z \rangle} F$ (and in particular any surface group) is fully residually free.
\end{prop}

Recall that a group is said to be fully residually free if 
for any finite set of non-identity elements in the group there is a homomorphism to a free group which is injective on that finite set.

\begin{remark}\label{rem:DehnTwist}
In the proof in \cite{wilton2005introduction} of the proposition above, for every finite~$X \subset G \setminus \{1\}$, there is a homomorphism~$h \colon G \rightarrow F$ such that~$1 \not \in h(X)$. The map~$h$ is actually the composition~$\rho \circ \delta_z^m$, where~$\rho \colon G \rightarrow F$ is the obvious retraction, and~$m \in \N$ is a sufficiently large integer depending on~$X$.

In particular, if $\Mgen{h(X)}$ is graded in $F$, then $\Mgen{X}$ is graded in $G$ with linear distortion; also, membership in $\Mgen{X}$ would be decidable.
\end{remark}

In the next example we show how the map $\rho \circ \delta_z^m$ in \cref{rem:DehnTwist} can be utilised to show that membership is decidable in certain submonoids of the surface group $S_2$. The interesting thing about the following example is that the generators do not satisfy the exponent sum conditions of Theorem~\ref{thm:PositiveSubonoidsOfSurfaceGroups} and hence that theorem cannot be applied.

\begin{example}\label{Ex:table} 
Let 
$S_2 = \Gpres{a,b,c,d}{aba^{-1}b^{-1} = dcd^{-1}c^{-1}}$
be the surface group of genus 2 and let 
$M = \Mgen{ab^{-2}, ba^{-2}, c^{-1}d^{2}, d^{-1}c^{2}} \leq S_2$. 
Observe that for each generator $x$ the generating words of $M$ include elements both with positive $x$-exponent sum and also elements with negative $x$-exponent sum. Thus       
the exponent sum condition of Theorem~\ref{thm:PositiveSubonoidsOfSurfaceGroups} is not satisfied. 

Then the map $h = \rho \circ \delta_z$ (where $z = aba^{-1}b^{-1} = dcd^{-1}c^{-1}$) 
defined in \cref{rem:DehnTwist} is the homomorphism 
from $S_2$ to $FG_{ \{ a,b \}}$ induced by
\[
  a \mapsto a, \, b \mapsto b, \, 
  c \mapsto (aba^{-1}b^{-1})b(bab^{-1}a^{-1}), \, 
  d \mapsto (aba^{-1}b^{-1})a(bab^{-1}a^{-1}).  
\]
Computing the images in $F(a,b)$ of the generators of $M$ under $h$ we obtain 
\[
h(M) = \Mgen{ab^{-2}, ba^{-2}, aba^{-1}b^{-2}a^{2}bab^{-1}a^{-1}, aba^{-1}b^{-1}a^{-1}b^{3}ab^{-1}a^{-1}}.
\]
Denote these generators of $h(M)$ by $\alpha_1, \alpha_2, \alpha_3, \alpha_4$. Note that their lengths are $3$ for the first two generators, and $11$ for the last two ones. Denote by $s_i$ the suffix of $\alpha_i$ starting from the letter in the middle; we have:
\[
s_1 = b^{-2},\, s_2 = a^{-2},\, s_3 = a^{2}bab^{-1}a^{-1},\, s_4 =  b^{3}ab^{-1}a^{-1}.
\]
\noindent {\bf Claim.} For any $1 \leq i, j \leq 4$ one has that $\text{red}(s_i \alpha_j) = p_{i, j}l_j$, where $p_{i, j}$ is a non-empty prefix of $s_i$, and $l_j$ is a suffix of $\alpha_j$ with $|l_j| \geq |s_j|$. 
\begin{proof}[Proof of Claim.] 

All the computations are done in the table in Figure~\ref{figtable}. 
\begin{figure} 
\[
\begin{array}{|c|c|c|}
\hline
s_i            & \alpha_j & \text{red}(s_i\alpha_j) \\
\hline
s_1=b^{-2}     & \alpha_1=ab^{-2} & s_1 \cdot \alpha_1 \\
s_1            & \alpha_2=ba^{-2} & b^{-1} \cdot s_2 \\
s_1            & \alpha_3 = aba^{-1}b^{-2}a^{2}bab^{-1}a^{-1}       & s_1  \cdot \alpha_3  \\
s_1            & \alpha_4 = aba^{-1}b^{-1}a^{-1}b^{3}ab^{-1}a^{-1}         & s_1 \cdot \alpha_4 \\
\hline
s_2=a^{-2}     & \alpha_1         & a^{-1} \cdot s_1 \\
s_2            & \alpha_2 & s_2  \cdot \alpha_2 \\
s_2            & \alpha_3         & a^{-1} \cdot ba^{-1}b^{-2}a^{2}bab^{-1}a^{-1} \\
s_2            & \alpha_4         & a^{-1} \cdot ba^{-1}b^{-1}a^{-1}b^{3}ab^{-1}a^{-1}\\
\hline
s_3=a^{2}bab^{-1}a^{-1} 
                & \alpha_1       & a^{2}bab^{-1} \cdot b^{-2}\\
s_3            & \alpha_2 & s_3 \cdot \alpha_2 \\
s_3            & \alpha_3         & a^{2} \cdot b^{-1}a^{2}bab^{-1}a^{-1} \\
s_3            & \alpha_4         & a \cdot b^{3}ab^{-1}a^{-1}\\
\hline
s_4=b^{3}ab^{-1}a^{-1} 
                & \alpha_1       & b^{3}ab^{-1}  \cdot b^{-2}\\
s_4            & \alpha_2 & s_4 \cdot \alpha_2 \\
s_4            & \alpha_3         & b \cdot a^{2}bab^{-1}a^{-1}\\
s_4            & \alpha_4         & b^{2} \cdot a^{-1}b^{3}ab^{-1}a^{-1} \\
\hline
\end{array}
\]
\caption{Calculations used for Example~\ref{Ex:table}.}
\label{figtable} 
  \end{figure}
The third column represents the reduced form of $s_i \alpha _j$, given as $p_{i, j} \cdot l_j$, where the pieces are as defined in the claim.
\end{proof}
It then follows easily from the claim that 
$l(\alpha_{i_1} \cdots \alpha_{i_k}) \geq k$, for any product of words from $\{\alpha_1, \alpha_2, \alpha_3, \alpha_4 \}$ which implies that $h(M)$ is graded. 

Indeed, for any product $f = \alpha_{i_1} \cdots \alpha_{i_k}$, its reduced form $g = \text{red}(f)$ in $F(a, b)$ ends with $w_fs_{i_k}$, where $w_f$ is a word of length at least $k$.
This holds for $k = 1$. Next, we use induction and the fact that
\[
\text{red}(\alpha_{i_1} \cdots \alpha_{i_k}) = \text{red}(\text{red}(\alpha_{i_1} \cdots \alpha_{i_{k-1}}) \cdot \alpha_{i_k}).
\]
By induction, $h = \text{red}(\alpha_{i_1} \cdots \alpha_{i_{k-1}})$ ends with $w_h s_{i_{k-1}}$, with $w_h$ a word of length at least $k-1$. Hence, $\text{red}(f)$ ends with $\text{red}(w_h s_{i_{k-1}} \alpha_{i_k})$, which ends with $w_h p_{i_{k - 1}, i_k} l_{i_k}$, as required.

Since $h(M)$ is graded it then follows from \cref{thm:DistortionMapToFreeGroup} that  
$M \leq S_2$ is graded and has decidable membership problem with linear distortion function.
From the claim note that in this case we can also show that $h(M)$ is a free monoid. 
\end{example}

\begin{remark}
If in the example above we put a power $k$ over every single letter we obtain an infinite family
    \[M_k = \Mgen{a^k b^{-2k}, b^ka^{-2k}, c^{-k}d^{2k}, d^{-k} c^{2k}}
    \]
for which in a similar way the same map $f \circ \delta_z$ can be used to prove that  
in all these submonoids of $S_2$ membership is decidable.
\end{remark}

\begin{remark}
A natural question arising from the results in this section is the following. Is it true that for every graded submonoid $M$ of the surface group $S_2$ there is a homomorphism $f:S_2 \rightarrow F_2$ to the free group of rank $2$ such that $f(M) \leq F_2$ is graded? If so this would show that membership in graded submonoids of $S_2$ is decidable. The proof of the previous example using the Dehn twist map offers an approach to this problem. We do not know whether the proof used to deal with that example can be extended to arbitrary graded submonoids of $S_2$ e.g. by taking a sufficiently large power of $\delta_z$. 
\end{remark}

\begin{example} Given the result Theorem~\ref{thm:PMPDistortionSurface} above for $S_2$ it is natural to ask to what extent this approach might be used to solve the membership problem for prefix monoids of other groups of the form $F \ast_{\langle w \rangle} F$ for a free group $F$. The following example tells us something about the limitations of this approach.    

Consider the group 
\[
  G = \Gpres{a,b,c,d}{aba^{-1}b^{-1}a^{-1} dcdc^{-1}d^{-1} = 1}
\]
We claim that Theorem~\ref{thm:S2matchedAdapted} cannot be applied to the prefix monoid this example.  

Indeed, consider the submonoid of $FG_{\{ a,b \}}$ generated by the prefixes of the word $aba^{-1}b^{-1}a^{-1}$. This is equal to the submonoid generated by 
\[
X = \{ a, ab, aba^{-1}, aba^{-1}b^{-1}, aba^{-1}b^{-1}a^{-1} \}. 
\]
Even though some of these generators are redundant we will work with respect to this generating set $X$. We claim that with respect to the generating set $X$ the submonoid $\Mgen{X} \leq FG_{ \{ a,b \}}$ is not graded. Then it follows from results in \cite{margolis2005distortion} that this monoid is not graded with respect to any finite generating set. With respect to the generating set $X$ observe that we have: 
\[
ab = (aba^{-1}b^{-1}a^{-1}) \cdot (ab) \cdot (a). 
\]
Generalising this we see that for every $k \in \mathbb{N}$ we have 
\[
ab = (aba^{-1}b^{-1}a^{-1})^k \cdot (ab) \cdot (a)^k. 
\]
Hence there are arbitrary long products of the generators $X$ that are all equal to $ab$. 
It follows that $\Mgen{X} \leq FG_{ \{ a,b \}}$ is not graded. 
So the prefix membership problem cannot be solved for this example using 
Theorem~\ref{thm:S2matchedAdapted}. 
Note however that the prefix membership problem for this example is decidable and that can be proved, for instance, 
by applying Theorem~\ref{thm: exponent_sum_one_relator2:NewCorrected} with $t=b$.  
\end{example}

Among other things, the previous example shows the limitations of the theory of graded monoids for studying the prefix membership problem. 

\subsection{Membership in submonoids of powers of generators}\label{sec: Membership on submonoids of powers of generators}

The following lemma is easy so we omit the proof. 

\begin{lemma}\label{lem: when monoids are groups}
Let~$S \subseteq \Z$ be a finite set, which contains both negative and positive elements. Then~$\Mgen{S}$ is an infinite cyclic group generated by~$d = \gcd(S)$.
\end{lemma}

\begin{cor}\label{cor:powers}
    Let~$P$ be a finite subset of powers of the generators 
    \[ \{a_1, a_1^{-1}, \ldots, a_g, a_g^{-1}, b_1, b_1^{-1}, \ldots, b_g, b_g^{-1} \} \]  
    of~$S_g$. Then membership in~$\Mgen{P}$ is decidable. 
\end{cor}
\begin{proof} We distinguish two cases:
\begin{itemize}
    \item[(i)] There is a generator~$s$ of~$S_g$ with~$\sigma_s(p) \geq 0$ (or~$\leq 0$) for all~$p \in P$. In this case we can use Theorem~\ref{thm:PositiveSubonoidsOfSurfaceGroups} to decide membership in~$\Mgen{P}$.
    \item[(ii)] If conditions of case (i) are not satisfied, then~$\Mgen{P}$ is a group (see Lemma \ref{lem: when monoids are groups}), and membership in subgroups of surface groups is decidable as they are LERF.
\end{itemize}
 \end{proof}

\subsection{Higher dimensions: 3-manifold groups}\label{sec:HighDim}
In general the submonoid membership problem for hyperbolic surface groups remain open. 
It is natural to ask what happens in dimensions above two. 
In particular can one classify the 3-manifold groups with decidable submonoid membership problem? 
We conclude the article with some very brief observations about this question including some infinite families of examples of such groups for which the submonoid membership problem is undecidable. 
In contrast all 3-manifold groups are known to have decidable subgroup membership problem \cite{friedl2016membership}.

Any right-angled Artin group $A(\Gamma)$ whose defining graph $\Gamma$ 
is a forest 
is known to be a 
a $3$-manifold group; see \cite{droms1987graph}.
Since the group $A(P_4)$ is known to have undecidable submonoid membership problem by \cite{lohrey2008submonoid}, for any forest $F$ containing the path $P_4$ the group $A(F)$ will be a $3$-manifold group with undecidable submonoid membership problem.
Further examples of 3-manifold groups with undecidable submonoid membership problem arise in the work of Niblo and Wise \cite{niblo2001} who show that if $M$ is a compact graph manifold such that $\pi_1(M)$ is not subgroup separable, then $\pi_1(M)$ contains $A(P_4)$ and hence has undecidable submonoid membership problem. 

\section*{Acknowledgements} 
We thank Yago Antol\'{i}n for useful conversations in particular suggesting Lemma~\ref{submonoids of free groups are un-distorted} and some ideas about its proof.  We also thank Marco Linton for helpful comments relating to both Remark~\ref{rmk:prank} and also Remark~\ref{rmk:maxmin}.

\end{document}